\documentclass[12pt]{amsart}
\usepackage{amsfonts,amssymb, amscd, latexsym, graphicx, psfrag, color, float}
\usepackage[all]{xy}

\newtheorem{dummy}{dummy}[section]
\newtheorem{lemma}[dummy]{Lemma}
\newtheorem{theorem}[dummy]{Theorem}

\newtheorem{conjecture}[dummy]{Conjecture}
\newtheorem{corollary}[dummy]{Corollary}
\newtheorem{proposition}[dummy]{Proposition}
\theoremstyle{definition}
\newtheorem{definition}[dummy]{Definition}
\newtheorem{example}[dummy]{Example}
\newtheorem{remark}[dummy]{Remark}

\newcommand{\bC}{\mathbb{C}}

\newcommand{\bN}{\mathbb{N}}
\newcommand{\bP}{\mathbb{P}}

\newcommand{\bR}{\mathbb{R}}
\newcommand{\bZ}{\mathbb{Z}}
\newcommand{\bT}{\mathbb{T}}

\newcommand{\cA}{\mathcal{A}}
\newcommand{\cB}{\mathcal{B}}
\newcommand{\cC}{\mathcal{C}}
\newcommand{\cE}{\mathcal{E}}
\newcommand{\cF}{\mathcal{F}}
\newcommand{\cG}{\mathcal{G}}
\newcommand{\cH}{\mathcal{H}}

\newcommand{\cL}{\mathcal{L}}

\newcommand{\cO}{\mathcal{O}}
\newcommand{\cP}{\mathcal{P}}
\newcommand{\cQ}{\mathcal{Q}}

\newcommand{\cS}{\mathcal{S}}

\newcommand{\cX}{\mathcal{X}}

\newcommand{\sfR}{\mathbb{C}}

\newcommand{\dghom}{\mathit{hom}}

\newcommand{\bSi}{\mathbf{\Sigma}}

\newcommand{\Sh}{\mathit{Sh}}

\newcommand{\Hom}{\mathrm{Hom}}
\newcommand{\Ext}{\mathrm{Ext}}
\newcommand{\Perf}{\mathcal{P}\mathrm{erf}}

\newcommand{\ltr}{\langle \Theta \rangle}

\newcommand{\ltrp}{\langle\Theta'\rangle}

\newcommand{\ori}{\mathfrak{or}}

\newcommand{\End}{\mathrm{End}}
\renewcommand{\mod}{\mathrm{mod}}
\newcommand{\forg}{\mathcal{F}org}
\newcommand{\lttr}{\langle \Theta_{T} \rangle}
\newcommand{\lttrp}{\langle \Theta_{T}' \rangle}
\begin{document}

\title[The nonequivariant coherent-constructible correspondence and tilting]{The nonequivariant coherent-constructible correspondence and tilting}

\begin{abstract}
The coherent-constructible correspondence is a relationship between coherent sheaves on a toric variety $X$ and constructible sheaves on a real torus $\bT$. This was discovered by Bondal, and explored in the equivariant setting by Fang, Liu, Treumann and Zaslow. 
In this paper we prove the equivariant coherent-constructible correspondence for a class of toric varieties including weighted projective space. Also, we give applications to the construction of tilting complexes in the derived category of toric DM stacks.
\end{abstract}

\author{Sarah Scherotzke}
\address{Sarah Scherotzke, 
Mathematical Institute of the University of Bonn, 
Endenicher Allee 60, 
53115 Bonn,
Germany}
\email{sarah@math.uni-bonn.de}

\author{Nicol\`o Sibilla}
\address{Nicol\`o Sibilla, Max Planck Institute for Mathematics, Vivatsgasse 7,
53111 Bonn,
Germany}
\email{sibilla@mpim-bonn.mpg.de}

\maketitle

{\small \tableofcontents}

\section{Introduction}

The coherent-constructible correspondence (\emph{CC correspondence}) is an equivalence of categories that relates coherent sheaves on a toric variety $X$, and constructible sheaves on a real torus $\bT$ which is dual to the compact torus acting on $X$. The CC correspondence was first conjectured by Bondal in the very influential preprint \cite{Bo}, and was established in the equivariant  setting by Fang, Liu, Treumann, and Zaslow in \cite{FLTZ1, FLTZ2}. In this paper we prove the non-equivariant CC correspondence for a class of toric varieties: our main application is the 
construction of tilting complexes for an interesting class of toric DM stacks. Below we 
give a more detailed sketch of our results.

\subsection{The non-equivariant coherent constructible correspondence}
Let $\mathbf{\Sigma} = (N, \Sigma, \beta)$ be a stacky fan (see Section \ref{sec:toric} for a definition), and let $M = Hom(N, \bZ)$ and $M_{\bR} = M \otimes_\bZ \bR$. Denote $\mathcal{X}_{\mathbf{\Sigma}}$ the corresponding toric orbifold. In \cite{FLTZ1, FLTZ2} the authors define a conical Lagrangian $\tilde{\Lambda}_{\bSi} \subset T^*M_{\bR}$, and construct an equivalence between the torus equivariant category of perfect complexes $\Perf_{T}(\cX_{\bSi})$, and the category of compactly supported constructible sheaves on $M_{\bR}$ with singular support in $\tilde{\Lambda}_{\bSi}$, $\Sh_c(M_{\bR}, \tilde{\Lambda}_{\bSi})$. The notion of singular support is due to Kashiwara and Schapira \cite{KS}, informally it is an invariant that measures the way sections of constructible sheaves propagate. We recall its definition and some of its basic properties in Section \ref{sec:constructible}.

The authors in \cite{FLTZ1, FLTZ2} conjecture that a similar statement holds when replacing the equivariant category with ordinary perfect complexes, $\Perf(\cX_{\bSi})$. For toric varieties partial results in this direction were obtained by Treumann in the preprint \cite{T}. In Section \ref{sec:main} we extend Treumann's work to toric orbifolds. 
More precisely, let $\bT$ be the quotient 
of $M_\bR$ by the integral lattice $M$, $\bT = M_{\bR}/M$. As $\tilde{\Lambda}$ is invariant under translation by $M$ we obtain a conical Lagrangian subset $\Lambda = \tilde{\Lambda}/M$ of $T^*\bT$. 

\begin{theorem} 
\label{thrm:main1}
There is a fully faithful functor $\kappa: Perf(\cX_{\bSi}) \rightarrow \Sh_c(\bT, \Lambda_{\bSi})$ which makes the diagram commute (up to natural equivalence)
$$
\xymatrix{
\Perf_{T}(\cX_{\bSi}) \ar[d]_{\mathcal{F}org} \ar[r]^{\tilde{\kappa}}_{\cong} & \Sh_c(M_{\bR}, \tilde{\Lambda}_{\bSi}) \ar[d]^{p_!} \\
\Perf(\cX_{\bSi}) \ar[r]^{\kappa} &  \Sh_c(\bT, \Lambda_{\bSi}),
}
$$
where $p:M_{\bR} \rightarrow M_{\bR}/M_\bZ$ is the quotient map, and $\mathcal{F}org$ is the forgetful functor.
\end{theorem}

If $\cX_\bSi$ is a toric orbifold, we say that the non-equivariant CC correspondence holds for $\cX_\bSi$ if   
$\kappa$ is an equivalence.
Following \cite{FLTZ1, FLTZ2} we make the following Conjecture.
\begin{conjecture}
\label{conj:coh-co}
The non-equivariant CC correspondence holds for all toric orbifolds.
\end{conjecture}

One of or main results is a proof of Conjecture 
\ref{conj:coh-co} for a class of stacky fans $\bSi$ 
called \emph{cragged} (see Definition \ref{def:cragged fans}), which satisfy some special combinatorial conditions. 
Bondal and Ruan announced that the 
non-equivariant CC correspondence holds for weighted projective space \cite{BR}: 
as weighted projective spaces have cragged fans, this follows from our work.

\begin{theorem}
\label{thrm:main1}
If $\bSi$ is a cragged stacky fan then the functor $\kappa: \Perf(\cX_{\bSi}) \rightarrow \Sh_c(\bT, \Lambda_{\bSi})$ is an equivalence.
\end{theorem}

Although the conditions defining cragged fans are quite restrictive they are satisfied in a number of interesting cases: we have already mentioned weighted projective spaces, also all toric Fano surfaces have cragged fans. 
The proof of Theorem \ref{thrm:main1} depends on a careful study of the category of constructible sheaves $\Sh_c(\bT, \Lambda_\bSi)$. When $\bSi$ is cragged $\Sh_c(\bT, \Lambda_\bSi)$ has some important properties, 
including the fact that it is closed under tensor product of constructible sheaves. This is one of the key ingredients in the proof of Theorem \ref{thrm:main1}. 

These properties of $\Sh_c(\bT, \Lambda_\bSi)$ follow from general results 
about constructible sheaves.  
In Section 
\ref{sec:constructible} we define   \emph{cragged Lagrangians}: if 
$\bSi$ is a cragged fan then  
$\Lambda_\bSi$ is a cragged Lagrangian in this sense. 
Using techniques from Kashiwara and Schapira's  microlocal sheaf theory \cite{KS} we prove that if $M$ is a real analytic manifold and $\Lambda \subset T^*M$ is cragged then $\Sh_c(M, \Lambda)$ is always closed under tensor product, and the restriction of the tautological $t$-structure of constructible sheaves restrict to a $t$-structure on this category.

\subsection{Tilting}
\label{sec:intro2}
Section \ref{section:tilting} contains an application of the non-equivariant CC correspondence to the algebraic geometry of toric DM stacks. 
In \cite{Ki} King conjectured that the derived category of a  
toric variety should contain a tilting object 
$\cE$, such that $\cE$ decomposes as a direct sum of line bundles. A lot of work has been done towards establishing King's Conjecture and its extension to toric orbifolds in several classes of examples, see for instance \cite{CM} and \cite{BH} and references therein. Although King's Conjecture is now known to be false in general \cite{HP, Ef} weaker statements are likely to be true:

\begin{conjecture}[\cite{Ef} Conjecture 1.6]
\label{conj:exceptional}
Let $\cX$ be a smooth and proper 
toric DM stack. Then $\Perf(\cX)$ admits 
a full strong exceptional collection.\footnote{Dropping the condition that the collection be strong, this has been proved by Kawamata \cite{Kaw}. Note also that if $\cX$ is Fano  Conjecture  \ref{conj:exceptional} follows from Dubrovin's Conjecture, which is motivated by Kontsevich's Homological Mirror Symmetry: see \cite{Bay} and references therein for more information on work in this area.}
\end{conjecture}

In Section \ref{section:tilting} we prove the following result, that gives new evidence for Conjecture \ref{conj:exceptional}.
\begin{theorem}
\label{thrm:main22}
If $\bSi$ is a cragged stacky fan then $\Perf(\cX_\bSi)$ admits a tilting complex. That is,  
$\Perf(\cX_\bSi)$ is quasi-equivalent to the derived category of a finite dimensional algebra $A$.
\end{theorem} 
The key point in the 
proof of Theorem \ref{thrm:main22} is the fact that as $\Lambda_\bSi$ is cragged, $\Sh_c(\bT, \Lambda_\bSi)$ comes together with a very natural $t$-structure and we have a good control of its heart. Theorem \ref{thrm:main22} does not give a complete proof of Conjecture \ref{conj:exceptional} for toric orbifolds with cragged fans:  we are currently unable to show that the algebra $A$ that is produced by our argument is directed.\footnote{$A$ is directed if and only if the derived category of $A$-modules admits a full strong exceptional collection.} However we believe that this is the case, we plan to investigate this issue in future work. 

{\bf Acknowledgements:} We thank David Treumann and Eric Zaslow for their interest in this project. We thank particularly David Treumann for many useful discussions and for letting us include in our paper some results from his 
preprint \cite{T}. N.S. thanks the Max Planck Institute for Mathematics for excellent working conditions.

\section{Conventions and background}
Throughout the paper we work over the field of complex numbers $\bC$. 
\subsection{Categories}
By \emph{dg category} we shall always mean a 
$\bC$-linear differential graded category, see \cite{Ke} and \cite{Dr} for 
definitions and basic properties. 
If $C$ is a dg category and $x, y$ are objects in $C$, we denote $hom_C(x,y)$ the complex of morphisms between $x$ and $y$.  We denote $D(C)$ the \emph{homotopy category} of the dg category $C$: 
$D(C)$ has the same objects then $C$, while the set of morphisms 
$Hom_{D(C)}(x, y)$ is given by the zero-th cohomology of the complex $hom_C(x, y)$. 
It will be often useful to work with 
\emph{triangulated} dg categories, which are defined for instance in Section 2.4 of \cite{Dr}.\footnote{Drinfeld calls such categories \emph{pre-triangulated}.} If $E$ is the set of objects in a triangulated dg category $C$ we denote $\langle E \rangle \subset C$ the smallest  triangulated dg subcategory of $C$ containing $E$. The homotopy category $D(C)$ of a triangulated dg category $C$ is triangulated in the classical sense. We say that a triangulated dg category $C$ has a t-structure if its homotopy category $D(C)$ has one. The reader can consult \cite{L} Section 1.2.1 for basic facts on t-structures in 
the context of stable $(\infty,1)$-categories: all the statements there can be adapted to  triangulated dg categories in a straightforward way. We sometimes consider diagrams of functors between 
dg categories: we say that such a diagram is \emph{commutative} if it is commutative up to canonical natural equivalence. 

For future reference we recall the notion of \emph{Postnikov systems} in 
triangulated dg categories, see Chapter $4$ of \cite{GM}. 
Postnikov systems give a convenient way of talking about resolutions in triangulated categories. 
\begin{definition}
Let $C$ be a triangulated dg category, and let 
$$
X^\bullet = X^c \stackrel{d^ c}{\rightarrow} X^{c+1} \stackrel{d^ {c+1}}{\rightarrow} \dots \stackrel{d^ 0}{\rightarrow} X^0
$$ 
be a sequence of objects in $C$ with degree zero maps between them, such that $d^{i+1}d^i$ = $0$ in the homotopy category.

\begin{itemize}
\item A \emph{(left) Postnikov system} attached to $X^\bullet$ is a diagram in $C$ of the form
$$
\xymatrix{ 
X^c \ar[r]^{d^c} \ar[rd]|{Id} & X^{c+1} \ar[r]^{d^{c+1}}  \ar[rd]|{g^{c+1}}& X^{c+2} \ar[r]^{d^{c+2}}  \ar[rd]|{g^{c+2}}&\cdots \ar[r]^{d^{-1}}  & X^0 \ar[rd]|{g0} \\ 
& Y^c = X^c  \ar[u]|{f^c = d^c} 
& Y^{c+1}  \ar[u]|{f^{c+1}} 
\ar[l]^{+1} & Y^{c+2}  \ar[u]|{f^{c+2}} 
\ar[l]^{+1} &\cdots  \ar[u]|{f^{-1}} 
\ar[l]^{+1} & Y^0,   \ar[l]^{+1} 
}
$$
which satisfies the following conditions:
\begin{enumerate}
\item  $f^ i$ and $g^ i$ have degree zero, and $f^ i g^ i$ = $d^ i$ in the homotopy category. 
\item $Y^ i \stackrel{f^ i}{\rightarrow} X^ {i+1} \stackrel{g^ {i+1}}{\rightarrow} Y ^{i+1}$ is a cofiber sequence.
\end{enumerate}
\item We call an object $Y$ of $C$ a \emph{left convolution} of $X^\bullet$ if there is a  Postnikov system attached to $X^\bullet$ with the property that $Y = Y^0$.
\end{itemize}
\end{definition}

\begin{remark}
Postnikov systems are preserved by dg functors of triangulated dg categories. 
Note also that if $Y \in C$ is a left convolution of $X^\bullet$ then $Y$ belongs to the subcategory $\langle X^c, \dots, X^0 \rangle \subset C$.
\end{remark}

\subsection{Constructible sheaves}
\label{sec:back-cons}
The classical reference for microlocal sheaf theory is \cite{KS}. 
Although the theory of \cite{KS} is cast in the language of ordinary 
triangulated categories, it is straightforward 
to adapt it to the setting of dg categories.  
This issue is discussed at greater length in 
\cite{NZ, N} to which we refer the reader for further details. If $M$ is a topological space   
we denote $\Sh(M)$ the 
dg category of bounded complexes of sheaves of 
$\bC$-vector spaces on $M$ localized at quasi-isomorphisms: $\Sh(M)$ is a dg enhancement 
of the usual derived category of sheaves over $M$, which is equivalent to the homotopy category $D(\Sh(M))$. For the theory of localization in the dg setting we refer the reader to \cite{Dr}. 

Assume that $M$ is a real analytic manifold. 
A sheaf $\cF$ of  $\bC$-vector spaces is called \emph{quasi-constructible} if
there is a Whitney stratification of $M$ such that the restriction of $\cF$ to each stratum is locally constant. $\cF$ is called \emph{constructible} if additionally its stalks $\cF_x$ are finite dimensional vector spaces for all $x \in M$. We denote $\Sh_c(M)$ (resp. $\Sh_{qc}(M)$) 
the triangulated dg category of complexes with constructible (resp. quasi-constructible) cohomology. We often refer to objects in $\Sh_c(M)$ simply as constructible sheaves. If $\cS$ is a Whitney stratification, we denote $\Sh_c(M, \cS) \subset \Sh_c(M)$ the triangulated dg category of complexes that have cohomology sheaves which are constructible with respect to $\cS$. 
An important invariant of an object $\cF \in \Sh_c(M)$ is its  
singular support, see \cite{KS} chapter V. We give a 
detailed review of this concept in Section \ref{sec:constructible}.   
If $\Lambda \subset T^*M$ is a 
$\bR_{>0}$-invariant Lagrangian, the triangulated dg category 
of constructible sheaves whose singular support is contained in
$\Lambda$ is denoted $\Sh_c(M, \Lambda) \subset \Sh_c(M)$. 

The theory of Grothendieck's (derived) six operations $f_*, f^*, f_!, f^!, \otimes, \cH om $ can be lifted to the dg category of constructible sheaves 
$\Sh_c(M)$, see \cite{N} Section 2.2. 
We denote $\Gamma(-)$ the functor of global sections. If  
$\cF \in \Sh_c(M)$ and $Z \subset M$ is a closed subset we denote   
$\Gamma_Z(\cF)$ the derived functor of the subsheaf of sections supported on $Z$. 
Throughout the paper, it will always be understood that all the functors between triangulated dg categories are dg derived: thus, for instance, we denote the derived push-forward simply $f_*$,  and not $Rf_*$, and similarly for $\Gamma(-)$ and $\Gamma_Z(-)$.

\subsection{Deligne-Mumford stacks}
An introduction to Deligne-Mumford (DM) stacks can be found 
in the appendix of \cite{V}. We refer to \cite{V} also for a definition of (quasi)-coherent sheaves and vector bundles on DM stacks. If $\cX$ is a DM stack we denote $\cQ Coh(\cX)$ the triangulated dg  
category of bounded complexes of quasi-coherent sheaves on $\cX$: $\cQ Coh(\cX)$ is obtained from the dg category of bounded complexes of quasi-coherent sheaves on $\cX$ by localizing at quasi-isomorphisms. A perfect complex is an object 
of $\cQ Coh(\cX)$ which is locally quasi-isomorphic to a complex of vector bundles 
on $\cX$. The triangulated dg category of perfect complexes is denoted $\Perf(\cX) \subset \cQ Coh(\cX)$. Let $G$ be an algebraic group acting on $\cX$ 
in the sense of \cite{Ro}. The $G$-equivariant categories of quasi-coherent sheaves and perfect complexes on $\cX$ will be denoted respectively 
$\cQ Coh_G(\cX)$ and $\Perf_G(\cX)$. 

\section{Toric orbifolds and stacky fans}
\label{sec:toric}
In this section we review the definition of \emph{toric DM stacks} and \emph{stacky fans} due to Borisov, Chen and Smith \cite{BCS}, see also \cite{FMN} 
and \cite{Iw}. We also recall a few facts in 
toric geometry which will be important later. 
We follow quite closely \cite{BCS} except we work in 
somewhat lesser generality. Indeed, in order to simplify the exposition we restrict to \emph{toric orbifolds}, i.e. toric DM stacks with trivial generic stabilizers. 

\begin{definition}
\label{def:stackyfan}
A \emph{(strict) stacky fan } $\bSi$ is given by a tuple $(N, \Sigma, \beta)$, where:
\begin{itemize}
\item $N \cong \bZ^n$ is a finitely generated 
free abelian group.
\item $\Sigma$ is a finite rational simplicial fan in $N_\bR = N \otimes_\bZ \bR$, such that its set of one dimensional cones $\Sigma(1) = \{\rho_1, \dots, \rho_r\}$ spans $N_\bR$. 
\item $\beta: \bZ^r = \bigoplus_{i=1}^{i=r}\bZ \cdot e_i \rightarrow N$ is a homomorphism such that for all $i \in \{1 \dots r \}$ $b_i := \beta(e_i)$ belongs to $\rho_i$.
\end{itemize}
\end{definition}

To any stacky fan we can attach a toric orbifold $\cX_\bSi$ in the following way. Let $M = Hom(N, \bZ)$ and let    
$\beta^*: M \rightarrow M':=Hom(\bZ^r, \bZ)$ be 
the dual of $\beta$. 
It follows from our assumptions that $\beta^*$ 
is injective. Denote $Gale(N)$ the cockernel of 
$\beta^*$, and set 
$$
G_\bSi := Hom(Gale(N), \bC^*).
$$  
The morphism $\beta^*$ induces an embedding  
$G_\bSi \hookrightarrow Hom(M', \bC^*) \cong (\bC^*)^r.$
For all $\sigma \in \Sigma$ denote $Z_\sigma \subset Hom(M', \bC) \cong \bC^r$ the closed subset defined as follows 
$$
Z_\sigma := \{(z_1, \dots, z_r) \in \bC^r |  \prod_{\text{ $i$ s.t. }\rho_i \nsubseteq \sigma} z_i = 0  \}.
$$
Let $U_\bSi$ be equal to the open subset $\bC^r -  \bigcup_{\sigma \in \Sigma}Z_\sigma$. Note that $U_\bSi$ is invariant under the action of $G_\bSi$. 
\begin{definition}
The toric orbifold $\cX_\bSi$ associated to the stacky fan $\bSi$ is the quotient stack $[U_\bSi/G_\bSi]$.
\end{definition} 
The fact that $\cX_\bSi$ is a DM stack is proved in 
Proposition 3.2 of \cite{BCS}. 
Also, it follows from Definition \ref{def:stackyfan} 
that $\cX_\bSi$ is an orbifold, i.e. it has trivial generic stabilizer.

\begin{remark} 
The coarse moduli space of $\cX_\bSi$ is the toric variety associated to the underlying ordinary fan $\Sigma$,  $X_{\Sigma}$, see \cite{BCS} Proposition 3.7.
If $\Sigma$ is smooth and all the vectors $b_i = \beta(e_i)$ are primitive, then $\cX_\bSi$ is 
a variety and is isomorphic to $X_\Sigma$. This is not the case if the $b_i$-s are primitive 
vectors but $\Sigma$ is not smooth: indeed $\cX_\bSi$ is always smooth.  
\end{remark}

\begin{remark}
As explained in \cite{FMN} toric DM stacks admit 
an intrinsic definition as compactifications of  DM tori (see \cite{FMN} Definition 3.2) carrying a DM torus action, much in the same way as ordinary toric varieties. The case of toric orbifolds 
is especially simple. Indeed, if the generic stabilizer is trivial 
the DM torus of \cite{FMN} 
coincides with the ordinary algebraic torus 
 $T = N \otimes_\bZ \bC^*$ and there is an open embedding $T \hookrightarrow \cX_\bSi$.
\end{remark}

Next we state some results  
which extend to toric orbifolds some familiar properties 
of toric varieties. In the sequel when referring to objects in $\Perf(\cX_\bSi)$ admitting an 
equivariant structure, this will always be with respect to the natural action by the open dense torus $T \hookrightarrow \cX_\bSi$.

\begin{proposition}[\cite{BH}]
\label{prop:resolution}
Let $\cX$ be a toric orbifold, then any coherent sheaf $\cF$ on $\cX$ admits a resolution $\mathcal{P}^{\bullet} = (\dots \rightarrow \mathcal{P}^1 \rightarrow \mathcal{P}^0) \rightarrow \mathcal{F}$ such that, for all $m$, $\mathcal{P}^m$ is isomorphic to a direct sum of line bundles: $\mathcal{P}^m \cong \bigoplus_{i=1}^{i = r} \cL_i$. 
\end{proposition}
\begin{proof}
This is proved in \cite{BH}, see Corollary 4.8.
\end{proof}

\begin{remark}
It will be useful to think of the resolution $\mathcal{P}^{\bullet}$ 
using the language of Postnikov systems. Indeed, we can regard $\cP^ i$ 
as objects in $\Perf(\cX)$ placed in degree zero. Then 
Proposition \ref{prop:resolution} can be restated as saying that $\cF$ is a 
left convolution of $\mathcal{P}^\bullet$.  
\end{remark} 

\begin{proposition}
\label{prop:line}
Any line bundle $\cL$ on a toric orbifold $\cX$ admits a $T$-equivariant structure.
\end{proposition}
\begin{proof}
Line bundles on $\cX_\bSi = [U_\bSi/G_\bSi]$ are given by $G_\bSi$-equivariant line bundles on $U_\bSi$. The statement then follows because $U_\bSi$ is an ordinary toric variety 
and line bundles 
on ordinary toric varieties always admit an equivariant structure.
\end{proof}

\begin{corollary} 
\label{cor:gen}
\begin{itemize}
\item Let $L$ be the set of line bundles on $\cX$, then $\langle L \rangle = \Perf(\cX)$.
\item Let $E_T$ be the set of objects in $\Perf(\cX)$ admitting an equivariant 
structure, then $\langle E_T \rangle = \Perf(\cX)$.
\end{itemize}
\end{corollary}

\section{Constructible sheaves and cragged lagrangians}
\label{sec:constructible}
In this section we prove a few general results about constructible sheaves which will play a key role in our argument. We start by 
giving some more explanations on the notion of singular support. For the rest of this section we assume that $M$ is a real analytic manifold. Let $\cF$ be an object in $\Sh_c(M)$:

\begin{definition}[\cite{KS} Definition 5.1.2]
The singular support $SS(\cF)$ of $\cF$ is the subset of $T^*M$ defined by the following condition: 
if $p$ is a point in $T^*M$, $p \notin SS(\cF)$ if and only if there exists an open neighborhood $U$ of $p$ having the property that for all $x_0 \in M$,
and for all $1$-differentiable function $\psi$ defined in a neighborhood of $x_0$ and such that 
$\psi(x_0) = 0$ and $d\psi(x_0) \in U$,
we have an isomorphism $(R\Gamma_{\{x \in M \text{ s. t. } \psi(x) \geq 0 \}}(\cF))_{x_0} \cong 0$.
\end{definition}

In Theorem \ref{thrm:ss-prop} we gather a few important properties of the singular support, referring to \cite{KS} Chapter V and VI for proofs. 
We say that a subset $A \subset T^*M$ is \emph{conical} if it is invariant under fiber-wise dilation by positive real numbers, that is if for all $z \in M$ and $\lambda \in \bR_{\ge 0}$
$\lambda (A \cap T_z^* M) \subset A$.

\begin{theorem}
\label{thrm:ss-prop}
Let $\cF \in \Sh_c(M)$, then 
\begin{enumerate}
\item $SS(\cF) \subset T^*M$ is a conical Lagrangian subset. 
\item The support of $\cF$ is equal to $SS(\cF) \cap T^*_MM$.
\item If $\cS$ is a Whitney stratification of $M$ and $\cF$ is constructible with respect to 
$\cS$, then $SS(\cF) \subset \bigcup_{S \in \cS} T^*_SM$.
\item If $\cF_1 \rightarrow \cF_2 \rightarrow \cF_3$ is a cofiber sequence 
in $\Sh_c(M)$, $SS(\cF_2) \subset SS(\cF_1) \cup SS(\cF_3)$.
\end{enumerate}
\end{theorem}

As we mentioned in Section \ref{sec:back-cons} if $\Lambda \subset T^*M$ is a conical Lagrangian set, we can consider the subcategory $\Sh_c(M, \Lambda)$ of 
 constructible sheaves $\cF$ such that $SS(\cF) \subset \Lambda$: Theorem \ref{thrm:ss-prop} (4) shows that $\Sh_c(M, \Lambda)$ is triangulated.  
\begin{remark}
Let $\cS$ be a Whitney stratification, and let $\Lambda_\cS = \bigcup_{S \in \cS} T^*_SM$. 
Then by Theorem \ref{thrm:ss-prop} (3), $Sh_c(M, \cS) \subset Sh_c(M, \Lambda_\cS)$. 
It is easy to see that this inclusion is an equivalence.
\end{remark}

The category $\Sh_c(M)$ has two additional structures which will be of interest for us: it is a symmetric monoidal category under tensor product of sheaves;  also,  
it comes equipped with a tautological t-structure whose  heart is given by sheaves placed in degree $0$.
However for a general conical Lagrangian $\Lambda$ the tensor product does not restrict to the subcategory 
$\Sh_c(M, \Lambda) \subset \Sh_c(M)$, and similarly for the  t-structure. 
In the rest of this Section we introduce a class of Lagrangians, called \emph{cragged}, having the property that both the t-structure and the tensor product restrict to $\Sh_c(M, \Lambda)$.  

\begin{remark}
Let $\cS$ be a Whitney stratification. It is easy to see that $\Sh_c(M, \cS)$ is closed under tensor product, and that the tautological t-structure of $Sh_c(M)$ restricts as well. Lagrangians of the form $\Lambda_\cS$ are encompassed by our notion of cragged Lagrangian. 
\end{remark}

We will need a notion of stratification which is more flexible than Whitney's. 
The exact properties that we require are specified below.
\begin{definition}
\label{def:stratification}
Let $M$ be a manifold. We call a \emph{stratification} of $M$ the datum of a preorder $I$, and a locally finite collection of locally closed subsets of $M$ indexed by $I$, 
$\{Z_i\}_{i \in I}$, with the following properties
\begin{itemize}
\item $\bigcup_{i \in I} Z_i = X$,
\item for all $i, j$, if $i \neq j$ then $Z_i \cap Z_j = \varnothing$,
\item $Z_i \cap \overline{Z_j} \neq \varnothing$ if and only if 
$Z_i$ is contained in $\overline{Z_j}$, and this holds if and only of $i \leq j$.
\end{itemize}
\end{definition}

\begin{definition}
\label{def:cragged}
\begin{enumerate}
\item Let $U$ be an open subset of a vector space $E$. We say that a closed subset 
$A \subset T^*U$ is \emph{globally cragged} if there is a stratification $\{Z_i\}_{i \in I}$ of $U$,
and a collection of closed convex cones $\{V_i\}_{i \in I}$ in $E^*$ such that:
$A = \bigcup_{i \in I} Z_i \times V_i$.

\item More generally, suppose that $M$ is a manifold and that $A \subset T^*M$ is a closed subset. We say that 
$A$ is \emph{cragged} if there is an atlas ${(U_\alpha, \phi_\alpha)}_{\alpha \in J}$ of $M$ such that, for all $\alpha \in J$,  
$(\phi_\alpha, d\phi_\alpha)$ is a trivialization of the cotangent bundle that maps $A_\alpha = A \cap T^*U_\alpha$ onto a globally cragged set.

\end{enumerate}
\end{definition}

\begin{lemma}
\label{lem:containment}
Under the assumptions of Definition \ref{def:cragged} (1), we also have that if $i \leq j$ then $V_i \supset V_j$.
\end{lemma}
\begin{proof} 
If $z$ lies in $Z_i$, there is a sequence $(x_n)_{n \in \bN}$ in $Z_j$ having the property that $x_n$ converges to $z$. Let $\xi$ be in $V_j$, and consider the sequence $(x_n, \xi)_{n \in \bN} \rightarrow (z, \xi)$. As $A$ is closed and $(x_n, \xi)_n$ belongs to $A$ for all $n$, the limit $(z, \xi)$ belongs to $A$ as well: thus $\xi$ lies in $V_i$. 
\end{proof}
Denote $+: T^*M \times_M T^*M \rightarrow T^*M$ the map of fiberwise addition. If $A$ and $B$ are subsets of $T^*M$, write $A + B$ for the image of $A \times_M B$ under $+$. Following \cite{KS} denote $A \hat{+} B$ the subset of $T^*M$ given in coordinates as follows:
 $(z, \zeta) \in T^*M$ lies in $A \hat{+} B$ if and only if there are sequences 
 $(x_n, \xi_n)_{n \in \bN}$ in $A$ and 
 $(y_n, \eta_n)_{n \in \bN}$ in $B$ such that
\begin{enumerate}
\item $x_n \stackrel{n} {\rightarrow} z$, $y_n \stackrel{n} {\rightarrow} z$, and $\xi_n + \eta_n \stackrel{n} {\rightarrow} \zeta$, and
\item $|x_n - y_n| |\xi_n|\stackrel{n} {\rightarrow} 0$.
\end{enumerate} 

\begin{definition}
A subset $A \subset T^*M$ is called \emph{pre-additive} (resp.\emph{additive}) if it is conical, and $A + A = A $  (resp. $A \hat{+} A = A$).  
\end{definition}
Note that an additive set is also pre-additive.  

\begin{lemma}
\label{lem:criterion}
Let $M$ be a manifold. If $A \subset T^*M$ is a cragged set,  
then $A$ is additive. 
\end{lemma}
\begin{proof}
Since additivity is a local property, we can assume that 
$M = U$ where $U$ is an open subset in a vector space $E$, 
and that 
$A$ is globally cragged.
We use the notation of Definition \ref{def:cragged} (1). If $p$ lies in $M$ we denote $i_p$ the element of $I$ such that $p$ belongs to $Z_{i_p}$.

Fix $z \in M$. We need to show that we have an inclusion $(A \hat{+} A)_z \subset A_z$. 
Note that for all points $p$ in a sufficiently small neighborhood of $z$, we have that $i_z \leq i_p$. 
Let $(z, \xi)$ be in $(A \hat{+} A)_z$. Then there are sequences 
$(x_n, \xi_n)_{n \in \bN}$ and $(y_n, \eta_n)_{n \in \bN}$ in $A$ such that $x_n \stackrel{n} {\rightarrow} z$, $y_n \stackrel{n} {\rightarrow} z$, and $\xi_n + \eta_n \stackrel{n} {\rightarrow} \zeta$. 
We can assume that $i_z \leq i_{x_n}$ and $i_z \leq i_{y_n}$ for all $n$. Since $A$ is cragged, this implies that for all $n$ both $(z, \xi_n)$ and $(z, \eta_n)$ 
belong to $A_z$. Also $A_z$ is a closed convex cone: since 
the sequence $(z, \xi_n + \eta_n)_{n \in \bN}$ lies in $A_z$, the limit $(z, \xi)$ to which this sequence converges belongs to $A_z$ as well. That is, $(z, \xi) \in A_z$  as we wanted to prove.
\end{proof}

Suppose that $M$, $A \subset T^*M$ and $\{U_\alpha, \phi_\alpha\}_{\alpha \in J}$ satisfy the conditions described in Definition \ref{def:cragged} (2) in the following weaker form. 
The chart $(\phi_\alpha, d\phi_\alpha)$ maps $A_{\alpha}= T^ *U_{\alpha} \cap A$ onto a set of the form $\bigcup_{i \in I} P_i \times V_i$, where $\{V_i\}_{i \in I}$ are 
convex cones, but we do not assume that $\{P_i\}_{i \in I}$ is a stratification: instead  
$\{P_i\}_{i \in I}$ is a locally finite collection of closed subsets such that $\bigcup_{i \in I}P_i = \phi_\alpha(U_\alpha) $. We will refer to this condition as 
$(2')$. 

Then the following holds:

\begin{lemma}
\label{lem:cragged} Let $A \subset T^*M$ be a closed subset. The following are equivalent: 
\begin{itemize}
\item $A$ is cragged. 
\item $A$ is pre-additive and satisfies condition $(2')$. 
\end{itemize}
\end{lemma}
\begin{proof}
Lemma \ref{lem:criterion} shows that cragged implies pre-additive, so we only have to prove the other implication. Let $\alpha \in J$: using $\phi_\alpha$ we identify $T^*U_\alpha$ and $U_\alpha \times E^*$, and we set  $U := U_\alpha$ and $A:= A \cap T^*U$.
%
By pre-additivity the fibers of $A$ over the points of 
$U$ are cones: also, they can be written as finite unions of elements in $\{V_i\}_{i \in I}$. By shrinking 
$U$ we can assume that, regarded as subsets of $E^*$,  there are only finitely many fibers of $A$ over $U$. We denote them $W_1, \dots, W_n$.  
By further shrinking $U$ we can also assume that the intersection of $\{P_i\}_{i \in I}$ with $U$ is a finite collection of closed subsets, which we denote 
$ P_1, \dots, P_m$.

Recall that a constructible set is a finite union of locally closed sets.
Given any partition of a space such that all 
its elements are constructible, we can refine it and get a stratification in the sense of Definition \ref{def:stratification}, see the proof of Theorem 18.11 in \cite{W}. Now consider the partition $Q_1, \dots, Q_n$ of $U$ defined by the following rule: $p \in Q_j$ if $A_p = W_j$. Note that the $Q_j$-s are obtained from the closed subsets $P_1, \dots, P_m$ by taking unions, intersections, and complements: thus, they are constructible. Let  
$\{Z_i\}_{i \in I'}$ be a stratification refining the partition $\{Q_1 \dots Q_n\}$. By construction, we can write $A = \bigcup_{i \in I'}Z_i \times W_i$, where $W_i \in \{ W_1, \dots, W_n\}$ are convex cones. That is, $A$ is cragged, as we wanted to show.
\end{proof}

\begin{lemma}
\label{lem:additive}
Let $\Lambda \subset T^*M$ be a conical Lagrangian subset and assume further that $\Lambda$ is additive. Then $\Sh_c(M, \Lambda)$ is closed under tensor product.
\end{lemma} 
\begin{proof}
This follows from results contained in Chapter 5 of \cite{KS} (see also Theorem 1.1 of \cite{GS}): in particular, we have an inclusion $SS(F \otimes G) \subset SS(F) \hat{+} SS(G)$. If $\Lambda$ is additive, then $SS(F) \hat{+} SS(G) \subset \Lambda \hat{+} \Lambda = \Lambda$.  
\end{proof}
\begin{corollary}
\label{cor:tensor}
If $\Lambda \subset T^*M$ is a cragged conical Lagrangian, then $\Sh_c(M, \Lambda)$ is closed under tensor product.
\end{corollary}

Next, we prove that if $\Lambda$ is a cragged conical Lagrangian the tautological $t$-structure on $Sh(M)$ restricts to a $t$-structure on $Sh(M, \Lambda)$. 

\begin{lemma}
\label{lem:loc}
Let $i: U \hookrightarrow M$ be an open subset, and let 
$\mathcal{F}$ and $\mathcal{G}$ be two objects 
in $Sh(M)$. Then:
\begin{itemize}
\item $SS(i^*\mathcal{F}) = SS(\mathcal{F}) \cap T^*U$.
\item $H^j(i^*\mathcal{F}) \cong i^*(H^j(\mathcal{F}))$.
\end{itemize} 
\end{lemma}
\begin{proof}
As the definition of singular support is local in nature, the first claim  follows immediately. 
The second is a consequence of the fact that $i^*$ is exact. 
\end{proof}

\begin{lemma}
\label{lem:exercise}
Let $U$ be an open subset of a vector space $E$, and let $V$ be a closed convex cone in $E^*$. If $\mathcal{F} \in Sh(M)$ has the property that $SS(\mathcal{F})$ is contained in $U \times V$, then $SS(H^j(\mathcal{F})) \subset U \times V$ for all $j$.
\end{lemma}
\begin{proof}
See \cite{KS}, Exercise VI.1.
\end{proof}

\begin{proposition}
\label{prop:cragged}
Let $M$ be a manifold, and let $\Lambda \subset T^*M$ be a cragged conical lagrangian. Then the 
$t$-structure on $Sh(M)$ restricts to a $t$-structure on $Sh(M, \Lambda)$.
\end{proposition}
\begin{proof}
We have to prove that for all objects $\mathcal{F}$ in $Sh(M, \Lambda)$, and for all $i \in \bZ$, $H^i(\mathcal{F})$ lies in $Sh(M, \Lambda)$, or equivalently that  
$SS(H^i(\mathcal{F})) \subset \Lambda$. By Lemma \ref{lem:loc} we can assume that $M$ is an open subset $U$ in a vector space $E$, and that $\Lambda$ is globally cragged. We use the notations of Definition \ref{def:cragged} (1).

If $p$ is a point in $U$ we denote $i_p \in I$ the index such that $p \in Z_{i_p}$. It is sufficient to show that for all $p$, the fiber $SS(H^ i(\mathcal{F}))_p$ is contained in $V_{i_p}$. 
We can consider an open neighborhood 
$U'$ of $p$ having the property that for all 
$q \in U'$, $i_p \leq i_q$. Let $i: U' \subset U$ be the inclusion.
By Lemma \ref{lem:loc} $SS(i^*\mathcal{F}) = SS(\mathcal{F}) \cap T^*U' \subset U' \times V_{i_p}$. Applying Lemma  \ref{lem:loc} and then Lemma \ref{lem:exercise} we obtain the inclusion 
$$ 
SS(H^j(\mathcal{F})) \cap T^*U' = SS(i^*H^j(\mathcal{F})) = 
SS(H^j(i^*\mathcal{F})) \subset U' \times V_{i_p},
$$
This implies in particular that $SS(H^j(\mathcal{F}))_p \subset V_{i_p}$.
\end{proof}

If $\cS$ is a stratification satisfying some extra conditions then $\Sh_c(M, \cS)$ can 
be described very explicitly as a category of quiver representations. 
We make use of the convenient recent reference \cite{B}, but remark that this circle of ideas goes back to McPherson and others:  for more information 
see for instance \cite{Ka}, and the discussion in the first Section of \cite{Tr1}.

Let us briefly summarize the results in \cite{B} that we will need in the sequel. 
Let $\cS$ be an \emph{acyclic stratification} of $M$ in the sense of \cite{B} Definition 17. For our purposes it is sufficient to remark that any triangulation of $M$ yields such a stratification, and that any Whitney stratification admits a refinement that is acyclic. Section 4 of \cite{B} contains a prescription that attaches to $\cS$ a directed quiver with relations $(Q_{\cS}, R_\cS)$. 
Let $A_\cS$ be the path algebra of 
$(Q_\cS, R_\cS)$, and let $D_{dg}^b(\mod A_\cS)$ be the triangulated dg category of finite $A_\cS$-modules.

\begin{remark}
Note that as the quiver $Q_\cS$ is directed, $A_\cS$ has finite global dimension. 
\end{remark}

\begin{theorem}[\cite{B} Theorem 21]
\label{theorem: quiver}
There is a quasi-equivalence $\Phi: \Sh_c(M, \cS) \cong D_{dg}^b( \mod A_\cS)$. 
Also, $\Phi$ preserves the tautological $t$-structures. 
\end{theorem}

\section{The coherent-constructible correspondence}
\label{sec:main}
We start by reviewing the statement of the 
equivariant coherent-constructible (CC) correspondence 
of \cite{FLTZ1, FLTZ2}. Next we extend to toric orbifolds partial 
results of Treumann \cite{T} on the non-equivariant CC-correspondence. 
Building on these results in the next Section we prove one of our main theorems: 
we establish a complete statement of the 
non-equivariant CC-correspondence  for a class of toric orbifolds satisfying 
some special combinatorial conditions. 

We use the notations introduced in Section \ref{sec:toric}. 
Let $\bSi = (N, \Sigma, \beta)$ be a stacky fan. Henceforth, we assume that all stacky fans $\bSi$ are complete: that is, that their underlying ordinary fan $\Sigma$ is complete.
Recall that we write $M = Hom(N, \bZ)$, $M_\bR = M \otimes_\bZ \bR$, and that $\beta$ is 
a homorphism of the form 
$\beta: \bZ^r = \oplus_{i=1}^{i=r}\bZ e_i \rightarrow N$. We set $b_i = \beta(e_i)$ 
and $B = \{b_1, \dots, b_r\}$.
For all $\sigma \in \Sigma$, we denote $N_{\sigma} \subset N$ the sublattice generated by $B \cap \sigma = \{w_1, \dots, w_d\}$, and we denote $M_{\sigma}$ the dual lattice $M_\sigma = Hom(N_{\sigma}, \bZ)$. If $\sigma \in \Sigma$, and $\chi \in M_\sigma$, we make the following definitions:
$$
\sigma^{\bot}_{\chi} := \{x \in M_{\bR} | \langle x, w_i \rangle = \langle \chi, w_i \rangle_{\sigma} \}, \text{     }
\sigma^\vee_\chi := \{x \in M_{\bR} | \langle x, w_i \rangle \geq \langle \chi, w_i \rangle_{\sigma} \},
$$
where $\langle \bullet , \bullet \rangle$ and $\langle \bullet , \bullet \rangle_{\sigma}$ are the natural pairings.
Following \cite{FLTZ2} we define the conical Lagrangian subset $\tilde{\Lambda}_\bSi \subset 
 M_\bR \times N_\bR = T^*M_\bR$ by the formula: 
$$
\tilde{\Lambda}_{\bSi} = \bigcup_{ \tau \in \Sigma} \bigcup_{ \chi \in M_{\tau}} \tau^{\bot}_{\chi} \times -\tau \subset M_{\bR} \times N_{\bR} = T^*M_{\bR}.
$$

Let $\Sh_{cc}(M, \tilde{\Lambda}_\bSi) \subset \Sh_{c}(M, \tilde{\Lambda}_\bSi)$ 
be the subcategory of sheaves with compact support. Denote $T = N \otimes_\bZ \bC^*$ 
the torus acting on $\cX$.
In \cite{FLTZ1} and \cite{FLTZ2} Fang, Liu, Treumann and Zaslow 
prove the following theorem:
\begin{theorem}[Equivariant CC-correspondence, \cite{FLTZ2} Theorem 7.6]
\label{thrm:eq-ccc}
There is a quasi-equivalence of dg categories
$
\tilde{\kappa}: \Perf_{T}(\cX_{\bSi}) \rightarrow \Sh_{cc}(M, \tilde{\Lambda}_{\bSi}).\footnote{In fact the theorem proved in \cite{FLTZ2} is more general as it applies as well to toric DM stacks with non-trivial generic stabilizers.}
$
\end{theorem}

The purpose of this Section is to extend work of Treumann \cite{T} on the non-equivariant CC-correspondence for toric varieties to the larger class of toric orbifolds. For the 
convenience of the reader we give complete proofs of all the statements. We do not claim any originality here as we will follow very closely the arguments in \cite{T}.  
Let $\bT$ be the torus obtained as the quotient $M_\bR / M$ and let $p$ be the quotient map, $p: M_\bR \to \bT$. Note that $\tilde{\Lambda}_\bSi$ is invariant under translation 
by elements in the integral lattice $M$. Denote $\Lambda_\bSi$ its quotient under the $M$-action: $\Lambda_\bSi$ is a conical Lagrangian subset of $T^* \bT$.
We denote $\forg$ both the forgetful  functor  
$\cQ Coh_T(\cX_{\bSi}) \to \cQ Coh(\cX_{\bSi})$,  
and its restriction to perfect complexes $\Perf_T(\cX_\bSi) \rightarrow \Perf(\cX_\bSi)$.
The following holds: 
\begin{theorem}
\label{thm: cohcon}
There is a quasi fully-faithful functor of triangulated dg categories 
$$
\kappa: \Perf(\cX_{\bSi}) \to \Sh_c(\bT, \Lambda_{\bSi}),
$$ 
which makes the following diagram commute:
$$
\xymatrix{
\Perf_T(\cX_{\bSi}) \ar[r]^{\tilde \kappa}\ar[d]_{\forg} & \Sh_{cc}(M_\bR, \tilde{\Lambda}_{\bSi}) \ar[d]^{p_!} \\
\Perf(\cX_{\bSi}) \ar[r]_\kappa & \Sh_c(\bT, \Lambda_{\bSi}).
}
$$
\end{theorem} 

We prove Theorem \ref{thm: cohcon} by constructing $\kappa$ explicitly on a set of generators. Let us review briefly the main steps of the Theorem \ref{thrm:eq-ccc}. In \cite{FLTZ2} the authors attach to all pairs $(\sigma, \chi)$,  
$\sigma \in \Sigma$ and $\chi \in M_\sigma$, a quasi-coherent sheaf $\Theta_{T}'(\sigma,\chi) \in \cQ Coh_T(\cX_{\bSi})$ and a constructible sheaf $\Theta_{T}(\sigma,\chi) \in \Sh_c(M_\bR)$ as follows: 
$$\begin{array}{ccc}
\Theta_{T}(\sigma,\chi) & := & ((\sigma_\chi^\vee)^\circ \hookrightarrow M_\bR)_! \ori[\dim(M_\bR)], \\
\Theta_{T}'(\sigma,\chi) & := & (\cX_\sigma \hookrightarrow \cX_\bSi)_* \cO_{\cX_\sigma}(\chi), 
\end{array}
$$ 
where $\cX_\sigma \subset \cX_\bSi$ is the open substack associated to the cone 
$\sigma$, and $(\sigma^\vee_\chi)^\circ \subset M_\bR$ denotes the interior of $\sigma^\vee_\chi$. 
The functor $\tilde \kappa$ is first defined as a functor between the larger categories 
$\cQ Coh(\cX_\bSi) \rightarrow \Sh_c(M_\bR, \tilde{\Lambda}_\bSi)$:
$\tilde{\kappa}$ sends $\Theta_{T}'(\sigma,\chi)$ to $ \Theta_{T}(\sigma,\chi)$ and this assignement determines $\tilde \kappa$ uniquely. The key fact is that $\tilde{\kappa}$ 
restricts to the subcategories $\Perf(\cX_\bSi)$ and $\Sh_{cc}(M_\bR, \tilde{\Lambda}_\bSi)$, and defines a quasi-equivalence between them.

In the non-equivariant setting we consider the sheaves 
$$
\begin{array}{ccc}
\Theta(\sigma, \chi) & := & p_! \Theta_{T}(\sigma,\chi), \\
\Theta'(\sigma, \chi) & := & \forg (\Theta_{T}(\sigma,\chi)).
\end{array}
$$
We start showing that the prescrition $\kappa (\Theta'(\sigma, \chi)) =\Theta(\sigma, \chi)$ defines a quasi 
fully-faithful embedding of dg categories $\kappa: \cQ Coh (\cX_\bSi) \rightarrow \Sh_{qc}(\bT, \Lambda_\bSi)$.
Note that we have a natural action of $M$ on $M_\sigma$: with slight abuse of notation, if $m \in M$ and $\chi \in M_\sigma$ we write $m + \chi$ for $(m)|_{N_\sigma} + \chi$. More generally, let $\tau$ and $\sigma$ be two cones such that $\tau \subset \sigma$, and thus also $\sigma^\vee \subset \tau^\vee $. We have a canonical map $M_{\sigma} \to M_{\tau} $ given by restriction to $N_\tau$. As before, if $\chi_1 \in M_\sigma$ and $\chi_2 \in M_\tau$ we simply write $\chi_1 + \chi_2$ instead of $(\chi_1)|_{N_\tau} + \chi_2$. 
We remark that the sheaves $\Theta(\sigma, \chi)$ and $\Theta'(\sigma, \chi)$ depend on $\chi$ only up to the action of $M$.

\begin{lemma}
\label{lemma: cones}
Let $\sigma$ and $\tau$ be cones in $\Sigma$. There is a $\xi \in M$ such that $\sigma_{\chi_1}^\vee \subset \tau_{\chi_2 +\xi}^\vee $ if and only if 
$\tau \subset \sigma$ and $ \xi \in \tau^\vee_{\chi_2-\chi_1} $. 
\end{lemma}
\begin{proof}
Let us assume first that $\tau \subset \sigma$. The condition $\xi \in  \tau^\vee_{\chi_2-\chi_1}$ implies that $\xi- \chi_1+\chi_2 \in \tau^{\vee} $ and hence $ \sigma^{\vee}  -\chi_1 +\chi_2+ \xi \subset \tau^{\vee}$.  
That is, 
$\sigma^{\vee}_{\chi_1} \subset \tau^{\vee}_{\chi_2 + \xi} $. 
Conversely, if $\sigma_{\chi_1}^\vee \subset \tau_{\chi_2 +\xi}^\vee $, then $\sigma^\vee -{\chi_1}+\chi_2 +\xi  \subset \tau^\vee $. It follows that 
$\sigma^{\vee} \subset \tau^{\vee}$, hence $\tau \subset \sigma$ and $-{\chi_1}+\chi_2 +\xi \in \tau^\vee $. The second condition is equivalent to the fact that 
$\xi \in  \tau^\vee_{\chi_2-\chi_1} $. 
\end{proof}

\begin{proposition} 
\label{prop:fully faithful}
We have 
$$
\Ext^i(\Theta(\sigma, \chi_1),\Theta(\tau, \chi_2)) \cong \bigg\{
\begin{array}{ll}
\sfR[\tau^\vee_{\chi_2-\chi_1} \cap M] & \text{if $i = 0$ and $\tau \subset \sigma$}, \\
0 & \text{otherwise}
\end{array}
$$
and
$$
\Ext^i(\Theta'(\sigma, \chi_1),\Theta'(\tau, \chi_2)) \cong \bigg\{
\begin{array}{ll}
\sfR[\tau^\vee_{\chi_2-\chi_1} \cap M] & \text{if $i = 0$ and $\tau \subset \sigma$}, \\
0 & \text{otherwise}
\end{array}
$$
\end{proposition}

\begin{proof}
By adjunction we can write
$$
\dghom(\Theta(\sigma, \chi_1),\Theta(\tau, \chi_2)) = \dghom(p_!\Theta_{T}(\sigma,\chi_1),\Theta(\tau, \chi_2)) \cong \dghom(\Theta_{T}(\sigma,\chi_1),p^! \Theta(\tau, \chi_2)).
$$
Also, $p^! \Theta(\tau, \chi_2)$ decomposes as the direct sum 
$\bigoplus_{\xi \in M} \Theta_{T}(\tau,\chi_2 + \xi)$.
Thus we have that 
$$
\dghom(\Theta(\sigma, \chi_1),\Theta(\tau, \chi_2)) \cong \bigoplus_{\xi \in M} \dghom(\Theta_{T}(\sigma,\chi_1),\Theta_{T}(\tau,\chi_2 + \xi)).
$$
Then the first equality is a consequence of \cite{FLTZ2} Proposition 5.5 and Lemma \ref{lemma: cones}. The second equality follows in a similar way.
\end{proof}

Let $\lttr \subset \Sh_{c}(M_\bR)$ and $\lttrp \subset \cQ Coh_T(\cX_{\bSi})$ denote the full triangulated dg subcategories generated respectively by the objects $\Theta_{T}(\sigma, \chi)$ and $\Theta_{T}'(\sigma, \chi)$.
In the non-equivariant setting, we write $\ltr \subset \Sh_{qc}(\bT)$ for the full triangulated dg subcategory generated by the objects  
$\Theta(\sigma, \chi)$, and $\ltrp \subset \cQ Coh(\cX_{\bSi})$ the one generated by the 
objects $\Theta'(\sigma,\chi)$. The assignment $\Theta'(\sigma, \chi) \mapsto \Theta(\sigma, \chi)$ determines uniquely a functor 
$\kappa: \ltrp  \rightarrow \ltr$.
\begin{theorem}
\label{thm:square}
The functor $\kappa:\ltrp  \to \ltr$ is a quasi-equivalence, and fits in the following commutative diagram of dg categories 
$$
\xymatrix{
\lttrp \ar[r]^{\tilde \kappa} \ar[d]_{\forg} & \lttr \ar[d]^{p_!} \\
\ltrp \ar[r]_{\kappa} & \ltr.
}
$$
\end{theorem}

\begin{proof}
The functor $\kappa$ is a quasi-equivalence by Proposition  \ref{prop:fully faithful}. The commutativity of the diagram follows 
from the definition of $\kappa$.  
\end{proof}

\begin{proposition}
The dg category $\Perf(\cX_{\bSi}) $ is contained in $\ltrp$. 
\end{proposition}

\begin{proof}
The equivariant CC-correspondence implies that 
$\Perf_T(\cX_\bSi) \subset \lttrp$. The image of $\Perf_T(\cX_\bSi)$ under the functor 
$\forg$  is by definition the set $E_T$ of perfect complexes on $\cX$ admitting an equivariant structure. It follows that $E_T \subset \ltrp$. Since $\ltrp$ is triangulated it 
must contain also the triangulated hull of $E_T$, $\langle E_T \rangle$. By Corollary 
\ref{cor:gen} $\langle E_T \rangle = \Perf(\cX_\bSi)$, and therefore 
$\Perf(\cX_\bSi)$ is contained in $\ltrp$: this concludes the proof.
\end{proof}

In view of the previous proposition we can restrict $\kappa$ to $\Perf(\cX_{\bSi})$. 
In order to prove Theorem \ref{thm: cohcon} it remains 
only to show that $\kappa$ sends $\Perf(\cX_{\bSi})$ to $\Sh_c(\bT, \Lambda_{\bSi})$:
this is the analogue of Theorem 7.1 of \cite{FLTZ2} in the non-equivariant setting. 
\begin{proof}[The proof of Theorem \ref{thm: cohcon}.]

Corollary \ref{cor:gen} shows that $\Perf(\cX_{\bSi}) = \langle L \rangle$ where $L$ is the set of line bundles on $\cX_\bSi$. It follows that the image of $\Perf(\cX_{\bSi})$ under 
$\kappa$ is equal to $\langle \kappa(L) \rangle $, where  $\kappa(L) := \{\kappa(\cL), \cL \in L\} $. 
By Proposition \ref{prop:line} if $\cL$ is in $L$ 
we can choose an equivariant representative $\tilde \cL \in \Perf_T(\cX_\bSi)$ 
such that $\cL \cong \forg(\tilde \cL)$. Also, using Theorem \ref{thm:square} 
we can write  
$ \kappa(\cL) \cong p_! (\tilde \kappa(\tilde{\cL}))$. 
One consequence of Theorem 7.1 of \cite{FLTZ2} is  
that $\kappa(\tilde{\cL})$ belongs to $\Sh_{cc}(M_\bR, \tilde{\Lambda}_\bSi)$. This implies that $p_!( \tilde \kappa(\tilde{\cL}))$ is a constructible sheaf 
and also that its singular support 
is contained in $p(\tilde\Lambda)$:  $p_!(\tilde \kappa(\tilde{\cL})) \in \Sh_c(\bT, \Lambda_{\bSi})$. 
Since $\Sh_c(\bT, \Lambda_{\bSi})$ is triangulated we conclude that
$$
\kappa( \Perf(\cX_\bSi) ) \subset \langle \kappa(L) \rangle \subset \Sh_c(\bT, \Lambda_{\bSi})
$$
as we wanted to prove.
\end{proof}

\begin{remark}
We believe that there is an inclusion $\Sh_c(\bT, \Lambda_{\bSi}) \subset \ltr$ but  we do not know how to prove this in full generality. Indeed, this statement is equivalent to the fact that the quasi-fully faithful embedding $\kappa: \Perf(\cX_\bSi) \rightarrow \Sh_c(\bT, \Lambda_\bSi)$ is also essentially surjective. In the next Section we give a proof of the essential surjectivity of $\kappa$ but only under additional assumptions on the combinatorics of $\bSi$.
\end{remark}

\subsection{Functoriality of $\kappa$ and tensor products}
\label{sec:functoriality}
For completeness we prove that $\kappa$ is functorial with respect to maps of stacky fans, 
and that $\kappa$ intertwines the ordinary tensor product on $\Perf(\cX_\bSi)$ and the convolution tensor product on $\Sh_c(\bT, \Lambda_\bSi)$. This is the analogue in the non-equivariant 
setting of Theorem 5.16 and Corollary 5.18 of \cite{FLTZ2}. We remark that these results will not be used in the rest of the paper.

The next Definition is due to Borisov, Chen and Smith, see \cite{BCS} Remark 4.5.
\begin{definition}
A \emph{map of stacky fans} between $\bSi_1 = (N_1, \Sigma_1, \beta_1)$ and $\bSi_2 = (N_2, \Sigma_2, \beta_2)$ is a group homomorphism $f : N_1 \rightarrow N_2$ with the following properties: 
\begin{enumerate}
\item Denote $f_\bR: (N_1)_\bR \rightarrow (N_2)_\bR$ the homomorphism obtained by extension of scalars. Then $f_\bR$ maps cones in $\Sigma_1$ to cones in $\Sigma_2$: that is, if $\sigma_1 \in \Sigma_1$, there is $\sigma_2 \in \Sigma_2$ such that $f_\bR(\sigma_1) \subset \sigma_2$.
\item If $\sigma_1 \in \Sigma_1$ and $\sigma_2 \in \Sigma_2$ are as in (1), then $f(N_{\sigma_1}) \subset 
N_{\sigma_2}$.
\end{enumerate}
\end{definition}

To a map of stacky fans $f$ we can attach (see \cite{FLTZ2} Section 5.6 for additional details):
\begin{itemize}
\item  A map of vector spaces 
$\tilde{v_f}: M_{2, \bR} \rightarrow M_{1, \bR}$. As $\tilde{v_f}^{-1}(M_{1, \bZ})$ is contained in $M_{2, \bZ}$, $\tilde{v_f}$ descends to a map $v_f: \bT_1:= M_{1, \bR} / M_{1, \bZ} \rightarrow \bT_2 := M_{2, \bR} / M_{2, \bZ}$.

\item A map of stacks $u_f: \cX_{\bSi_1} \rightarrow  \cX_{\bSi_2}$  which is equivariant with respect to the torus actions: that is, $u_f$ fits in the commutative diagram
$$
\xymatrix{
T_1 \times \cX_{\bSi_1} \ar[r]^{u_f|_{T_1} \times u_f } \ar[d]_{m_1} &  T_2 \times \cX_{\bSi_2} \ar[d]^{m_2} \\
\cX_{\bSi_1} \ar[r]_{u_f} & \cX_{\bSi_2},
}
$$ 
where $T_1 \subset \cX_{\bSi_1}$ and  
$T_2 \subset \cX_{\bSi_2}$ are the open torus orbits, 
and $m_1$ and $m_2$ denote the respective actions on $\cX_{\bSi_1}$ and $\cX_{\bSi_2}$.
\end{itemize}

\begin{theorem}
\label{thm:functoriality}
Let $f: \bSi_1 \rightarrow \bSi_2$ be a map of stacky fans. 
Assume that the underlying homomorphism $f:N_1 \rightarrow N_2$ is injective. Then we have a commutative diagram of dg categories: 
$$
\xymatrix{
\langle \Theta'_{\bSi_2} \rangle \ar[r]^{\kappa_2} \ar[d]_{u^*}  & 
\langle \Theta_{\bSi_2} \ar[d]^{v_!} \rangle \\
\langle \Theta'_{\bSi_1} \rangle \ar[r]_{\kappa_1} & \langle \Theta_{\bSi_1} \rangle.
}
$$

\end{theorem}

\begin{proof}
We define a natural transformation $\iota :v_! \circ \kappa_2 \to \kappa_1 \circ u^*$  on the generators $\Theta'(\sigma, \chi)$.  Note that there are 
natural quasi-isomorphisms
\begin{itemize}
\item 
$
v_! \circ \kappa_2(\Theta'_{\bSi_2}(\sigma, \chi)) \cong v_! \Theta_{\bSi_2}(\sigma, \chi) = v_! p_{2!} \Theta_{T_2, \bSi_2}(\sigma, \chi) \cong p_{1!} \tilde{v}_! \Theta_{T_2, \bSi_2}(\sigma, \chi) \cong \\
 \cong p_{1!} (\tilde v_! \tilde \kappa_2 (\Theta'_{\bSi_2}(\sigma, \chi))),
$
\item 
$ 
\kappa_1 \circ u^*(\Theta'_{\bSi_2}(\sigma, \chi)) \cong \kappa_1 \circ \forg(u^* \Theta_{T_2, \bSi_2}'(\sigma, \chi)) 
\cong
p_{1!} (\tilde \kappa_1 u^* \Theta_{T_2, \bSi_2}'(\sigma, \chi) ). 
$
\end{itemize}
The natural transformation $\iota$ can therefore be defined by 
applying $p_{1!}$ to the natural transformation $\tilde \iota:v_! \circ \tilde \kappa_2 \to \tilde \kappa_1 \circ u^*$ which is defined in \cite{FLTZ2} Theorem 5.16.  As $\tilde \iota$ is a quasi-isomorphism,  $\iota$ is as well, and this concludes the proof.
\end{proof}

\begin{corollary}
\label{cor: tensor}
The functor $\kappa: \Perf(\cX_\bSi) \rightarrow \Sh_c(\bT, \Lambda_\bSi)$ is monoidal with respect to the usual tensor product of perfect complexes, and the convolution product on constructible sheaves. 
\end{corollary} 
\begin{proof} 

Let $\bSi \times \bSi$ be the stacky fan of the 
product $\cX_\bSi \times \cX_\bSi$. 
The diagonal  
$\Delta: N \rightarrow N \times N$ defines a injective morphism 
of stacky fans between $\bSi$ and $\bSi \times \bSi$. 
Note that $u_f$ is the diagonal embedding $\cX_{\bSi} \to \cX_{\bSi} \times \cX_{\bSi}$, while $v_f$ is the addition map:
$v_f: \bT \times \bT \to \bT, v_f(a, b) = a +b$.  
Let $\cF, \cG$ be objects in $\ltrp_\bSi$ (resp., in $\ltr_\bSi$): 
the 
\emph{box product} of $\cF$ and $\cG$ is defined by the formula
$\cF \boxtimes \cG := p_1^* \cF \otimes p_2^* \cG \in \ltrp_{\bSi \times \bSi}$ (resp. $\cF \boxtimes \cG \in \ltr_{\bSi \times \bSi}$), 
where 
$p_1, p_2$ are the projections, $p_1, p_2: \cX_\bSi \times \cX_\bSi \rightarrow \cX_\bSi$ 
(resp. $p_1, p_2: \bT \times \bT \rightarrow \bT$). 
It follows immediately from the definition of $\kappa$ that there are natural 
isomorphisms $\kappa(\Theta'(\sigma_1, \chi_1) \boxtimes \Theta'(\sigma_2, \chi_2)) \cong \kappa(\Theta'(\sigma_1, \chi_1)) \boxtimes \kappa(\Theta'(\sigma_2, \chi_2))$.
 
Denote $\otimes^o$ the ordinary tensor product on $\ltrp$, and $\otimes^c$ the convolution tensor product on $\ltr$. 
Applying Theorem \ref{thm:functoriality} to $\Delta$ we obtain natural 
isomorphisms
$$\kappa(\Theta'(\sigma_1, \chi_1) \otimes^ o \Theta'(\sigma_2, \chi_2) ) \cong \kappa ( u_f^*  (\Theta'(\sigma_1, \chi_1) \boxtimes \Theta'(\sigma_2, \chi_2) )) \cong$$
$$ 
v_{f!} \kappa ( \Theta'(\sigma_1, \chi_1) \boxtimes  \Theta'(\sigma_2, \chi_2) ) = v_{f!} (\Theta(\sigma_1, 
\chi_1) \boxtimes \Theta(\sigma_2, \chi_2)) \cong 
\Theta(\sigma_1, 
\chi_1) \otimes^ c \Theta(\sigma_2, \chi_2).$$ 
This implies that $\kappa: \ltrp \rightarrow \ltr$ intertwines 
$\otimes^o$ and $\otimes ^c$. As these tensor structures 
restrict to the subcategories  
$\Perf(\cX_\bSi)$ and $\Sh_c(\bT, \Lambda_\bSi)$ this proves the statement.
\end{proof}

\section{Cragged stacky fans and the non equivariant cc-correspondence}
\label{sec:cragged-fans}

In this Section we prove that the quasi fully-faithful functor 
$\kappa: \Perf(\cX_\bSi) \rightarrow \Sh_c(\bT, \Lambda_\bSi)$ 
is essentially surjective provided $\bSi$ satisfies some special properties, that is, 
it is \emph{cragged} in the sense of Definition \ref{def:cragged fans}. 
 
\begin{remark}
\label{rem:equivalence}
The essential surjectivity of $\kappa$ is equivalent to the statement that the image of $p_{!}$ generates $\Sh_c(\bT, \Lambda_{\bSi})$. 
Indeed Corollary \ref{cor:gen} shows that the image of $\mathcal{F}org$ generates $\Perf(\cX_{\bSi})$, and $\tilde{\kappa}$ is an equivalence by the equivariant CC-correspondence of \cite{FLTZ2}. 
\end{remark}

\begin{lemma}
\label{lem:skyscraper}
Let $\cX$ a toric DM stack, consider the point $1 \in T \subset \cX$, and let $k(1)$ be the skyscraper sheaf at $1$. Then,  $\kappa(k(1)) \cong \bC_{\bT}$. 
\end{lemma}
\begin{proof}
Denote $j: T \hookrightarrow \cX$ the inclusion of the open torus orbit. Computing the Koszul resolution of $k(1)$ gives a complex 
\begin{equation}
\label{eq1}
(j_*\mathcal{O}_T)^{\binom{n}{n}} 
\rightarrow \dots \rightarrow (j_* \cO_T)^{\binom{n}{2}} \rightarrow (j_* \cO_T)^n \rightarrow  j_* \cO_T, 
\end{equation}
which is isomorphic to $k(1)$ in the derived category of quasi-coherent sheaves on $\cX$.
Recall that $\kappa$ is defined precisely by assigning quasi-constructible sheaves to quasi-coherent sheaves obtained as push-forward of structure sheaves on torus invariant open subsets.
Thus, we can evaluate $\kappa$ on the complex (\ref{eq1}).

The functor $\kappa$ admits a 
simple geometric interpretation on complexes of $j_*\cO_T$-modules. Indeed, we can regard a 
complex of $j_*\mathcal{O}_T$-modules as a complex of modules over the group algebra $\bC[\pi_1(\bT)] = \bC[T_1, T_1^{-1}, \dots, T_m, T_m^{-1}]$. 
Further, the abelian category of modules over 
$\bC[\pi_1(\bT)]$ is naturally equivalent to the category of 
locally constant quasi-constructible sheaves on $\bT$: $\kappa$ maps a complex of  
$j_*\mathcal{O}_T$-modules to the corresponding 
complex of locally constant quasi-constructible sheaves on $\bT$.

As a complex of $\bC[\pi_1(\bT)]$-modules, (\ref{eq1}) 
is the Koszul resolution of the trivial representation of 
$\bC[\pi_1(\bT)]$,
\begin{equation}
(\bC[\pi_1(\bT)])^{\binom{n}{n}} 
\rightarrow \dots \rightarrow (\bC[\pi_1(\bT)])^{\binom{n}{2}} \rightarrow (\bC[\pi_1(\bT)])^n \rightarrow \bC[\pi_1(\bT)].
\end{equation}
Thus $\kappa(k(1))$ is the locally constant sheaf given by the trivial $\bC[\pi_1(\bT)]$-module, i.e. the constant sheaf $\bC_\bT$. This concludes the proof.
\end{proof}

Let $\mathbf{\Sigma} = (N, \Sigma, \beta)$ be a stacky fan. 
Let $\beta$ be of the form $\beta: \bZ^r = \oplus_{i=1}^{i=r}\bZ e_i \rightarrow N$. As in Section \ref{sec:main} we denote $b_i = \beta(e_i)$ 
and $B = \{b_1, \dots, b_r\}$. 
\begin{definition}
\label{def:cragged fans}
We say that $\bSi$ is \emph{cragged} if the following two conditions 
are satisfied:
\begin{enumerate}
\item (exhaustiveness) Let $S$ be a subset of $\Sigma$. Denote $\langle S \rangle \subset N_{\bR}$ the smallest convex cone containing all cones $\sigma$ in $S$. Then there is a subset $T \subset \Sigma$ such that $\langle S \rangle =  \bigcup_{\tau \in T}\tau$.
\item (unimodularity) Let $b_{1}, \dots, b_{l}$ be a set of linearly independent vectors in $B$, and let $\rho_i = \bR_{\geq 0}b_i$. Exhaustiveness implies that 
$\langle \rho_1, \dots ,\rho_l \rangle = \bigcup_{\tau \in T}\tau$. Let $N_T \subset N$ be the sublattice generated by $B \cap \bigcup_{\tau \in T}\tau$. Then $b_{1}, \dots, b_{l}$ form a $\bZ$-basis for $N_T$.
\end{enumerate}
\end{definition}

\begin{proposition}
\label{prop:additive}
A stacky fan $\bSi$ is cragged if and only if $\tilde \Lambda_\bSi$ is a cragged Lagrangian.
\end{proposition}
\begin{proof}

We start by proving that if $\bSi$ is cragged then $\Lambda_\bSi$ 
is a cragged Lagrangian. Note that 
for all stacky fans $\mathbf{\Sigma}$, $\tilde{\Lambda}_\mathbf{\Sigma}$ satisfies the assumptions of Lemma \ref{lem:cragged}. Thus it is sufficient  to prove that if $\bSi$ is exhaustive and unimodular then 
$\tilde{\Lambda}_\bSi$ is pre-additive. 
We have to show that for all $\phi \in M_{\bR}$, $(\tilde{\Lambda}_{\bSi})_\phi$ is a convex cone. 
Denote $S_{\phi}$ the 
subset of $\Sigma$ consisting of all cones $\tau$ such that $\phi(N_{\tau}) $ takes integral values,  that is such that $\phi \in M_{\tau}$.

 By exhaustiveness, there is a $T_{\phi} \subset \Sigma$ such that $\langle S_{\phi} \rangle = 
\bigcup_{\tau \in T_\phi}\tau$. 
Using unimodularity we can find a subset $B_\phi \subset B$ with the  
following properties:
\begin{enumerate} 
\item the vectors in $B_\phi$ form a basis of the sublattice generated by $B \cap T_\phi$,
\item for each $b \in B_\phi$ there is a $\sigma \in S_\phi$ such that $b$ belongs to $\sigma$.
\end{enumerate}
This implies that $T_\phi \subset S_\phi$. Indeed,
let $\tau \in T_{\phi}$ and set $B^\tau = B \cap \tau$. 
The cone $\tau$ lies in $S_\phi$ if and only if $\phi$ takes integral values on the elements of $B^\tau$. Since $B^\tau$ is contained in the sublattice generated by $B \cap \bigcup_{\tau \in T_\phi}\tau$, all its elements can be written as integral linear combinations of vectors in 
$B_\phi$. By definition $\phi$ takes integral values on $B_\phi$, therefore the same is true for $B^\tau$ as well.
These considerations give us an inclusion 
$$
-\langle S_\phi \rangle = \bigcup_{\tau \in T_{\phi}}-\tau \subseteq  \bigcup_{\sigma \in S_{\phi}}-\sigma = (\tilde{\Lambda}_{\bSi})_\phi,
$$ 
and the reverse inclusion is clear. Thus $(\tilde{\Lambda}_{\bSi})_\phi = -\langle S_\phi \rangle$, and therefore $(\tilde{\Lambda}_{\bSi})_\phi$ is a convex cone as we needed to show.

Let us prove the other implication: if $\tilde{\Lambda}_\bSi$ is pre-additive then 
$\bSi$ is exhaustive and unimodular. Let $S$ be a subset of $\Sigma$, denote $N_S$ the sublattice of $N$ generated by $\langle S \rangle \cap B$ and denote $M_S$ the dual lattice. Pick a minimal set of generators $\phi_1^S, \cdots, \phi_l^S$ of $M_S$. 
Choose preimages $\phi_1, \cdots, \phi_l$ in $M_\bR$ such that $\phi_i$ is sent to $\phi_i^S$ under the quotient map $M_\bR \rightarrow M_S\otimes_\bZ\bR $. 
The fibers of $\tilde\Lambda_\bSi$ are naturally subsets of $N_\bR$. Thus it makes sense to consider the intersection of the fibers 
$(\tilde{\Lambda}_\bSi)_{\phi_i}$ as a subset of $N_\bR$. It is immediate to see that 
$\cap_{i}(\tilde{\Lambda}_\bSi)_{\phi_i} = \bigcup_{\sigma \subset \langle S \rangle}-\sigma$. Since $\tilde{\Lambda}_\bSi$ is pre-additive, $(\tilde{\Lambda}_\bSi)_{\phi_i}$ is a convex cone for all $i$, and therefore $\cap_{i}(\tilde{\Lambda}_\bSi)_{\phi_i} =\bigcup_{\sigma \subset \langle S \rangle}-\sigma$ is a 
convex cone as well. 
There is an inclusion $\bigcup_{\sigma \subset \langle S \rangle}\sigma \subset \langle S \rangle$. The reverse inclusion holds as well since $\bigcup_{\sigma \in S } \sigma \subset \bigcup_{\sigma \subset \langle S \rangle}\sigma $, and by definition $\langle S \rangle$ is the smallest convex cone containing $\bigcup_{\sigma \in S } \sigma$. Thus $\langle S \rangle = \bigcup_{\sigma \subset \langle S \rangle}\sigma$: this proves that $\bSi$ is exhaustive.

Next we show that unimodularity holds as well. Suppose to the contrary that unimodularity fails: in the notation of the statement this means that there are one dimensional cones $\rho_{1}, \dots ,\rho_m$ such that $\langle \rho_1, \dots ,\rho_m \rangle = \bigcup_{\tau \in T}\tau$ but $\{b_1, \dots, b_m\}$ is not a basis of the sublattice $N_T$  generated by $B \cap \bigcup_{\tau \in T}\tau$. We define a homomorphism 
$\phi: N \rightarrow \bR$ having the following two properties:
\begin{enumerate}
\item $\phi$ takes integer values on $\{b_1, \dots, b_m\}$, 
\item there is an element $b_0 \in B \cap \bigcup_{\tau \in T}\tau$ such that $\phi(b_0)$ is not an integer.
\end{enumerate} 
Denote $\rho_0$ the one dimensional cone such that $b_0 \in \rho_0$. The conditions on $\phi$ imply that $-\rho_i$ is contained in $(\tilde{\Lambda}_\bSi)_\phi$ for all $\rho \in \{ 1, \dots, n\}$, while $-\rho_0$ is not: $-\rho_0 \nsubseteq (\tilde{\Lambda}_\bSi)_\phi$. 
Since $\Lambda_\bSi$ is pre-additive $(\tilde{\Lambda}_\bSi)_\phi$ is a convex cone and therefore contains $-\langle \rho_1, \dots ,\rho_m \rangle$ as a convex sub-cone. 
This gives a contradiction because then we would have  
$-\rho_0 \subset -\langle \rho_1, \dots ,\rho_m \rangle  \subset (\tilde{\Lambda}_\bSi)_\phi$. 
\end{proof}

We give below two examples of fans where 
either unimodularity or exhaustiveness fails.

\begin{example}

 \begin{itemize}
  \item 
Consider the fan $\Sigma$ in $\bR^3$ given by 
$\langle 0 \rangle$, the one-dimensional cones $\rho_i$, the four faces $\langle \rho_1, \rho_2 \rangle, \langle \rho_1, \rho_3 \rangle, \langle \rho_2, \rho_4 \rangle, \langle \rho_3, \rho_4 \rangle$ and the cone spanned by all the one-dimensional cones $\langle \rho_1, \rho_2, \rho_3, \rho_4 \rangle$. 
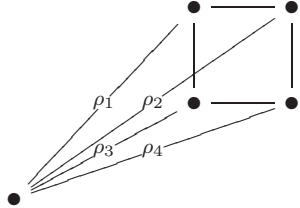
\[ 
\xymatrix{ 
&& \bullet \ar@{-}[r] \ar@{-}[d] & \bullet \ar@{-}[d] \\
&& \bullet \ar@{-}[r] & \bullet \\
 \bullet \ar@{-}[rruu] | {\rho_1} \ar@{-}[rrruu]| {\rho_2}  \ar@{-}[rru] | {\rho_3} \ar@{-}[rrru] | {\rho_4} & & & } \]
Then $\Sigma$ is not exhaustive: indeed the cone $\langle \rho_1, \rho_2, \rho_3 \rangle$ cannot be written as a union of cones of $\Sigma$. 

\item All complete fans in $\bR^2$ are exhaustive, but they might fail to be unimodular: 
in fact unimodularity is a more restrictive condition than smoothness. Consider the smooth complete fan $\Sigma$ in $\bR^2$ spanned by the one-dimensional cones $\rho_1 :=\bR_{>0} (1,0)$, $\rho_2 :=\bR_{>0} (1,1)$,  
$\rho_3:=\bR_{>0}(1, -1)$ and $\rho_4:=\bR_{>0}(-1, 0)$: 
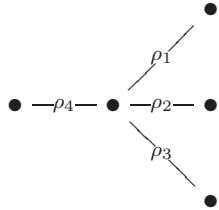
\[ \xymatrix{ & & \bullet \\
\bullet \ar@{-}[r] | {\rho_4} & \bullet   \ar@{-}[ru] | {\rho_1} \ar@{-}[r]| {\rho_2}  \ar@{-}[rd] | {\rho_3} & \bullet \\
& & \bullet }
\]
If $i, j$ are in $\{1, \ldots, 4\}$ 
denote $\sigma_{ij}$ the two dimensional cone of $\Sigma$ given by $\sigma_{ij} = \langle \rho_i, \rho_j \rangle$.
Then $\langle \rho_1 ,\rho_3 \rangle = \sigma_{12} \cup \sigma_{23}$ but the $\bZ$-lattice spanned by $N \cap  (\sigma_{12} \cup \sigma_{23})$ is generated by $(1,1), (1, -1)$ and $(1, 0)$. Hence 
$\Sigma$ is not unimodular. 
\end{itemize}
\end{example} 

\begin{corollary}
$\bSi$ is cragged if and only if $\Lambda_\bSi$ is a cragged Lagrangian.
\end{corollary}
\begin{proof}
For Lagrangians being cragged is a local condition, 
thus $\Lambda_\bSi$ is cragged if and only if $\tilde \Lambda_\bSi$ is cragged. This together with Proposition \ref{prop:additive} proves the statement.
\end{proof}

The conditions of Definitions \ref{def:cragged fans} are quite restrictive   but they are verified in a number of interesting examples. For instance, all toric Fano surfaces have cragged fans. Also Propositions \ref{prop:fwps} and \ref{prop:quotients} below imply, in particular, that in all dimensions there are infinitely many examples of toric orbifolds with cragged fans.

\begin{definition}
Let $\bSi = (N, \Sigma, \beta)$ be a stacky fan such that $N$ has rank 
$n$ and $\Sigma$ contains $n+1$ one dimensional cones. We say that $\cX_\bSi$ is a \emph{fake weighted projective space}.
\end{definition}

\begin{remark}
Equivalently, fake weighted projective spaces are toric orbifolds with Picard number $1$. 
Fake weighted projective spaces contain ordinary weighted projective 
spaces but form a strictly larger class, see \cite{Bu} for several concrete examples. 
\end{remark}

\begin{proposition}
\label{prop:fwps}
If $\bSi = (N, \Sigma, \beta)$ is the stacky fan of a fake weighted projective space then $\bSi$ is cragged.
\end{proposition} 
\begin{proof}
Let $N = \bZ^n$. Then $\bSi$ is fully determined by the assignment of a set of $n+1$ vectors $B = \{b_1 \dots b_{n+1}\}$ in $N = \bZ^n$: they satisfy an equation of the form $m_1b_1 + \dots + m_{n+1}b_{n+1} = 0$ where all the coefficients are strictly positive integers. The one dimensional cones of $\Sigma$ are given by $\rho_i = \bR_{\geq 0}b_i$, $\beta: \bigoplus_{i=1}^{i=n+1}\bZ e_i \rightarrow N$ maps $e_i$ to $b_i$, and the fan $\Sigma$ is the collection of all the cones spanned by proper subsets of $\{\rho_1 \dots \rho_{n+1}\}$. This implies in particular that $\bSi$ is exhaustive. 
Now let $b_{i_1} \dots b_{i_l}$ be a set of linearly independent vectors in $B$. By definition  
$\langle \rho_{i_1} \dots \rho_{i_l} \rangle$ is a cone in $\Sigma$. Thus the fact that $b_{i_1} \dots b_{i_l}$ are a $\bZ$-basis of the sublattice generated by $B \cap \langle \rho_{i_1} \dots \rho_{i_l} \rangle  = \{ b_{i_1} \dots b_{i_l} \}$ is tautological: $\bSi$ is therefore unimodular as well. 
\end{proof}

\begin{proposition}
\label{prop:quotients}
Let $\cX_\bSi$ be a toric orbifold with cragged stacky fan 
$\bSi = (N, \Sigma, \beta)$, and let $G$ be a finite group acting on $\cX$ in a way that is compatible with the torus action. Then the stacky fan of the quotient
$[\cX_\bSi/G]$ is cragged as well. 
\end{proposition}
\begin{proof}
Note that $G$ can be embedded as a finite subgroup of the torus $T = N \otimes_\bZ \bC^*$ in such a way that the action of $G$ on $\cX_\bSi$ 
is induced from the action of $T$ by restriction. 
The subgroup $G$ determines an overlattice $M'$ of $M = Hom(N, \bZ)$: $M'$ is the lattice of characters of $T$ which restrict to the trivial 
representation of $G$.  By dualizing we get 
a surjective map 
$q: N \rightarrow N'$. The stacky fan of $[\cX_\bSi/G]$ is given by  
$\bSi' = (N', \Sigma', q \circ \beta)$, where $\Sigma'$ is the 
set of cones $\{ q(\sigma) | \sigma \in \Sigma \}$. Exhaustiveness and unimodularity of $\bSi'$ follow immediately from the same properties of 
$\bSi$.
\end{proof}

The next Theorem is one of our main results.
\begin{theorem}
\label{thrm:main2}
If $\cX_{\bSi}$ is a smooth toric orbifold such that $\bSi$ is cragged, then
$$
\kappa: \Perf(\cX_{\bSi}) \rightarrow \Sh_c(\bT, \Lambda_{\bSi})
$$
is an equivalence of categories.
\end{theorem}
\begin{proof}
As explained  
in Remark \ref{rem:equivalence} it is sufficient to prove that the image of 
$$
p_!: \Sh_{c}(M, \tilde \Lambda_{\bSi}) \rightarrow \Sh_c(\bT, \Lambda_\bSi)
$$ 
generates $\Sh_c(\bT, \Lambda_\bSi)$.
Informally, our argument depends on the availability of a kind of categorification of partitions of unity: we explain briefly this analogy 
before proceeding with the proof. 
The role of the unit function 
is played by the constant sheaf $\bC_{\bT} \in \Sh_c(\bT, \Lambda_\bSi)$, which is the unit for the tensor product\footnote{We stress that here we are working with the ordinary tensor product of constructible sheaves, as opposed to the \emph{convolution} product which was discussed in Section \ref{sec:functoriality}.} in $\Sh_c(\bT, \Lambda_\bSi)$. We regard the image under 
$\kappa$ of a line bundle $\cL$ as corresponding to a bump function (or rather, to  the indicator function of an open subset): 
as a justification recall that if $\tilde \cL$ is an equivariant ample line bundle then $\tilde{\kappa}(\cL) = i_{!}\bC_{\Delta_{\cL}^{\circ}}$, where $\bC_{\Delta_{\cL}^{\circ}}$ is the constant sheaf on the interior of the lattice polytope $\Delta_{\cL}$ (see \cite{FLTZ2}). When $\bSi$ is cragged we 
can construct a complex $X^\bullet$ in $\Sh_c(\bT, \Lambda_\bSi)$ with the property that $X^i$ are direct sums of objects of the form $\kappa(\cL)$ and $\bC_T$ is a left convolution of $X^\bullet$: this is the key step in our argument, and we think of this as an  analogue of a partition of unity.

Lemma \ref{lem:skyscraper} 
implies that there is a resolution 
$$
\cP^\bullet = (\cP^m \rightarrow \cP^{m-1} \rightarrow \dots \rightarrow \cP^0) \rightarrow k(1) 
$$ 
of the skyscraper sheaf $k(1)$ such that $P^i$ is a direct sum of line bundles 
for all $i$. We regard 
$\cP^\bullet$ as a complex of objects in $\Perf(\cX_\bSi)$ and $k(1)$ as the left convolution of $\cP^\bullet$. Evaluating 
$\kappa$ on $\cP^\bullet$ we obtain a complex in $\Sh_c(\bT, \Lambda_{\bSi})$, $\kappa(P^\bullet) = (\kappa(\cP^m) \rightarrow \dots \rightarrow \kappa(\cP^0) )$  such that $\kappa(k(1)) \cong \bC_{\bT}$ is its left convolution.  
We can use this to realize all objects of $\Sh_c(\bT, \Lambda_\bSi)$ as iterated cones of objects lying in the image of $p_!$ in the following manner.

Let $\cF \in Sh(\bT, \Lambda_{\bSi})$. We tensor 
$\kappa(P^\bullet)$ with 
$\cF$ and we get a complex
$$
\kappa(\cP^m) \otimes \cF \rightarrow \dots \rightarrow \kappa(\cP^0) 
\otimes \cF,    
$$
such that $\cF = \bC_\bT \otimes \cF$ is its left convolution. 
This implies that $\cF$ lies in the subcategory 
$$
\langle \kappa(\cP^m) \otimes \cF, \dots, \kappa(\cP^0) \otimes \cF \rangle \subset \Sh_c(\bT).
$$ 
Note that $\kappa(\cP^i) \otimes \cF$ is  
an object of $\Sh_c(\bT, \Lambda_\bSi)$ for all $i$: in fact, since $\Lambda_\bSi$ 
is cragged, $\Sh_c(\bT, \Lambda_\bSi)$ is closed under tensor product 
by Corollary \ref{cor:tensor}.
We prove next that $\kappa(\cP^i) \otimes \cF$ 
lies in the image of $p_!$.
Write $\mathcal{P}^i = \bigoplus_{j \in J} L_j$ where $L_j$ are line bundles, and choose equivariant representatives $\tilde L_j \in Perf_{T}(\cX_{\bSi})$.  
There are natural isomorphisms 
$
\kappa(\bigoplus_{j \in J}\mathcal{L}_j) \cong \bigoplus_{j \in J}\kappa(\mathcal{L}_j) \cong \bigoplus_{j \in J} p_!(\tilde{\kappa}(\tilde{\mathcal{L}}_j)).
$
Thus we have that
$$
\mathcal{P}^i \otimes \mathcal{F} \cong \bigoplus_{j \in J}(p_!(\tilde{\kappa}(\tilde{\mathcal{L}}_j)) \otimes \mathcal{F}) \cong \bigoplus_{j \in J} 
p_!(\tilde{\kappa}(\tilde{\mathcal{L}}_j) \otimes p^! \mathcal{F}),
$$
where the last isomorphism is given by the projection formula.\footnote{Note that $p^! \mathcal{F} \in \Sh_c(M, \tilde{\Lambda}_{\bSi})$ is not compactly supported: however the tensor product $\tilde{\kappa}(\tilde{\mathcal{L}}_j) \otimes p^! \mathcal{F}$ is, and therefore lies in $\Sh_{cc}(M, \tilde{\Lambda}_{\bSi})$ as required.}
 This completes the proof.   
\end{proof}

\section{Tilting complexes}
\label{section:tilting}
As an application of Theorem 
\ref{thrm:main2} we prove that if $\bSi$ is a cragged 
stacky fan, then $\Perf(\cX_\bSi)$ has a tilting complex. 
Let $A$ be a finite-dimensional algebra over 
$\bC$.
\footnote{All algebras appearing in this paper are always assumed to be 
associative and unital.} 
Let us recall some notations:  
$\mod A$ is the abelian category of finite dimensional 
$A$-modules, $D^b_{dg}(A)$ is the triangulated dg category of 
finite dimensional
$A$-modules, and $D^b( A)$ the homotopy category 
$D^b(A) = D(D^b_{dg}(A))$. $D^b_{dg}( A)$ has a tautological $t$-structure with heart equal to 
$\mod A$. 

\begin{theorem}
\label{theorem: derived}
Let $A$ be a finite dimensional algebra of finite global dimension. 
Let $\cC$ be a triangulated dg category. Assume that there is a 
quasi fully-faithful embedding of $\cC$ into $D^b_{dg}(A)$ 
that satisfies the following conditions:
\begin{enumerate}
\item The restriction of the tautological $t$-structure of
$D^b_{dg}( A)$ is a well defined $t$-structure on $\cC$.
\item The embedding $\cC \rightarrow D^b_{dg}(A)$ 
admits a left adjoint.
\end{enumerate}
Then there exists a finite-dimensional algebra $B$ 
having global dimension smaller or equal than the global dimension of $A$, and a quasi-equivalence $\cC \cong D^b_{dg}( B)$. 
\end{theorem}
Let $A$ and $\cC$ be as in the statement of the theorem. 
We denote $\cC^0 = \cC \cap \mod A$: $\cC^0$ is 
the heart of the induced $t$-structure on $\cC$. The homotopy category $D(\cC)$ is a thick subcategory of $D^b(A)$, and the inclusion 
$\cC \rightarrow D^b_{dg}(A)$ admits a left adjoint which we denote $L$. 

\begin{lemma}
\label{homology}
An object $X \in D^b_{dg}(A)$ belongs to $\cC$ if and only if $H^i(X) \in \cC^0$ for all $i \in \bZ$. 
\end{lemma}
\begin{proof}
Clearly if $X \in \cC$ then $H^i(X) \in \cC^0$. 
We prove the other implication by induction on the number  
of indexes $i$ for which $H^i(X)$ is different from zero.
Suppose that $H^i(X)$ vanishes for all but one $i$. Then $X \cong H^i(X)$ and as the objects of $\cC^0 $ are contained in $\cC$ we have that $X \in \cC$. 
Suppose now that the set of indices $i$ such that $H^i(X) \not =0$ 
has cardinality $n$. Take $j$ to be the maximal index such that $H^j(X) \not =0$. In $D^b_{dg}(A)$ we have a cofiber sequence  
$Y \to X \to H^j(X)$, where $Y$ has non-vanishing cohomology 
in a strictly smaller set of degrees than $X$. By the inductive hypothesis 
$Y \in \cC$. As $\cC$ is triangulated, $X$ belongs to $\cC$ as well. 
\end{proof}

Let $\cA$ and $\cB$ be two abelian categories. We say that 
an exact functor $F : \cB \to \cA$ is \emph{degree-wise fully faithful} if 
$\Ext^i_{\cB}(A, B) \cong \Ext^i_{\cA}(F A, F B)$ for all $i \in \bN$. 

\begin{lemma}
\label{extension} 
Let $A$ and $B$ be finite-dimensional algebras. Suppose that $A$ has finite global dimension. 
If $F: \mod B \to \mod A$ is a degree-wise fully faithful functor, then 
its extension to the dg enhancements $F: D^b_{dg}(B) \to D^b_{dg}( A)$ is also quasi fully faithful. Further the global dimension of $B$ is smaller or equal than the global dimension of $A$. 
\end{lemma}
\begin{proof}
Since $F$ is degree-wise fully faithful, the $A$-module $F(B)$ satisfies $\Ext^i_A(F( B), F( B))=0$ for all $i \not =0$. It follows from  \cite{Rickard} Theorem 2.12 and Propositions 3.1 and 3.2 that 
$F$ induces a fully faithful embedding into $D^b(A)$ of the 
triangulated category of perfect $B$-modules,  
$\Perf(\mod B) $. 
Note that the maximum range of non vanishing 
$Ext$-groups between $B$-modules is bounded by that of $A$-modules. That is, the global dimension of $B$ is bounded by the global dimension of $A$, and $\Perf(\mod B) \cong D^b(B)$. Thus, there is a fully faithful embedding $D^b(B) \rightarrow D^b(A)$. 
By the uniqueness of cg enhancements (see \cite{LO}[Proposition 2.6]) for $D^b(B)$ and $D^b(A)$,  the embedding lifts to a quasi fully faithful embedding $F:D^b_{dg}( B) \to D^b_{dg} (A)$.  This completes the proof.
%
\end{proof}
Let $(D_{\le 0 }, D_{>0} )$ be the aisles of the tautological $t$-structure on $D^b(A)$, and $\tau_{\le 0} : D^b(A) \to D_{\le 0}$ and $\tau_{>0}: D^b (A) \to D_{> 0}$ be the truncation functors. 

\begin{proposition}
\label{truncation}
Let $D(C)$ and $D^b(A)$ be as in the statement of Theorem \ref{theorem: derived}. 
Suppose now that $M \cong \tau_{\le 0} M$ for some $M \in \cC$, that is $M$ has zero homology only in positive degree. Then the left adjoint $L(M)$ has zero homology in positive degree, that is $L(M) \cong \tau_{\le 0} L(M)$. 
\end{proposition}
\begin{proof}
If $M$ is an object in $D^b( A)$ there is a unique object $M^{\perp} \in D(\cC)^{\perp}$ 
such that $M^{\perp} \to M \to L(M) \to M^{\perp}[1]$ is a distinguished triangle. 
We have the following commutative diagram whose middle row and columns are distinguished triangles:
\[ 
\xymatrix{ \tau_{\le 0} M^{\perp} \ar[r] \ar[d]& \tau_{\le 0} M \ar[r] \ar[d] &  \tau_{\le 0} L(M) \ar[d] \\
\ M^{\perp} \ar[r] \ar[d] & M \ar[r] \ar[d] &  L(M) \ar[d] \\
\tau_{>0} M^{\perp} \ar[r]  & \tau_{>0} M \ar[r] &  \tau_{> 0} L(M).}
\]
Suppose now that $M \cong \tau_{\le 0} M$. Then $\tau_{>0} M$ vanishes and $\Hom_{D^b(A)}(M, X)$=$0$ for all $X$ in $D_{>0}$. As $ \tau_{>0} L(M)$ lies in $\cC$, it follows that 
$$
\Hom_{D^b(A)}(M, \tau_{>0} L(M))\cong \Hom_{D^b(A)}(L(M), \tau_{>0} L(M)) = 0.
$$  
Thus $ \tau_{>0} L(M) = 0$, and we have that $L(M) \cong \tau_{\le 0} L(M)$. 
\end{proof}
We can now proceed with the proof the main theorem of this section.
\begin{proof}[The proof of Theorem \ref{theorem: derived}]
Denote by $j$ the embedding of $\mod A$ into $D^b( A)$ 
as the heart of the tautological $t$-structure.
Consider the following commutative diagram: 

\[ \xymatrix{D( \cC) \ar[r]_i  \ar[d]_{H^0} & D^b( A)  \ar@/_/[l]_L  \ar[d]_{H^0} \\ 
\cC^0 \ar[r]_{i_0}  & \mod A  \ar@/_/[u]_j \text{    .}}
\] 
We denote $L^0: \mod A \to \cC^0$ the composition $L ^0 = H^0 L j$. 
We prove first that $L^0$ is a left adjoint of $i_0$. 
Let $N \in \mod A$ and $ M \in \cC^0$. We have that $H^0L(N)=\tau_{> -1} \tau_{\le 0} L(N)$. Since $N$ is concentrated in degree zero by Remark \ref{truncation}, $\tau_{\le 0} L(N)\cong L(N)$. 
Thus $$\Hom_{\cC^0} (L^0 N, M ) \cong \Hom_{D^b_{dg}(A)}( \tau_{>-1} L (N),   M).$$ 
We apply $\Hom(-, M)$ to the distinguished triangle 
$$\tau_{\le -1} L(N) \to L(N) \to \tau_{>-1} L(N) \to \tau_{\le -1} L(N)[1].$$
As $M$ is concentrated in degree zero, we deduce that
$$\Hom_{D^b(A)}( \tau_{>-1} L N,   M) \cong \Hom( L N, M) \cong \Hom(N, M),$$ that is, $L^0$ is a left adjoint of $i_0$. 

Left adjoints of exact functors preserve projective generators. 
Thus $L^0(A)$ is a projective generator of $\cC^0$, and 
$\cC^0$ is equivalent to $\mod B$ where $B$ is equal to $\End_{\mod A}(L^0(A))$. The equivalence $\mod B \longrightarrow \cC^0 $ is given by the functor 
$-\otimes_B L^0(A)$. Consider the composition $F:=i_0( -\otimes_B L^0(A))$, $F$ is a degree-wise fully faithful functor from $\mod B$ to $\mod A$. By Lemma \ref{extension}, $F$ extends to a fully faithful embedding $F: D^b_{dg}(B) \to D^b_{dg}( A)$ whose image consists of complexes with 
cohomology in $\cC^0$. By Lemma \ref{homology},  
the image of $F$ lies in $\cC$. Conversely $D(\cC)$ is generated by $\cC^0$ as a triangulated category. This shows that there is an equivalence 
$\cC \cong D^b_{dg}(B)$. 
\end{proof} 

\begin{theorem}
Let $\cX_\bSi$ be a smooth toric orbifold such that $\bSi$ is cragged. 
Then there is a finite-dimensional algebra $B$ of finite 
global dimension such that $\Perf(\cX_{\bSi})$ is quasi-equivalent to 
$D^b_{dg}( B)$. In particular $\Perf(\cX_{\bSi})$ contains a tilting complex.

\end{theorem}
\begin{proof}
Since $\Lambda_\bSi$ is cragged, the tautological $t$-structure of $\Sh_c(\bT)$ restricts to $\Sh_c(\bT, \Lambda_{\bSi})$.  
Let $\cS$ be an acyclic Whitney stratification of $\bT$ such that all complexes in  
$\Sh_c(\bT, \Lambda_\bSi)$ have cohomology that is constructible with respect to $\cS$:
this gives an embedding $\Sh_c(\bT, \Lambda_\bSi) \subset \Sh_c(\bT, \cS)$. By Theorem \ref{theorem: quiver} there is a quasi-equivalence $\Sh_c(\bT, \cS) \cong D^b_{dg}(A_\cS)$ that preserves the $t$-structures. 
Thus we obtain a quasi fully-faithful functor 
$$
J: \Sh_c(\bT, \Lambda_\bSi) \rightarrow D^b_{dg}(A_\cS),
$$
that is compatible with the $t$-structures. 
Also, $J$ has a left adjoint. Indeed, since $\cX_\bSi$ is a smooth and  
proper DM stack, $\Perf(\cX_\bSi)$ is a saturated dg category in the sense of \cite{To}. In particular, all quasi fully-faithful functors with source $\Perf(\cX_\bSi)$ admit both a left and a right adjoint, and the same holds for the quasi-equivalent category $\Sh_c(\bT, \Lambda_\bSi)$. This completes the proof, as we can apply 
Theorem \ref{theorem: derived} and obtain a chain of quasi-equivalences 
$$
\Perf(\cX_\bSi) \cong \Sh_c(\bT, \Lambda_\bSi) \cong D^b_{dg}(B).
$$
\end{proof}

\end{document}